\pdfoutput=1
\documentclass[oneside,fleqn,reqno,11pt]{amsart}
\usepackage[utf8]{inputenc}
\usepackage[T1]{fontenc}
\usepackage{etex}
\usepackage[a4paper]{geometry}
\usepackage{lmodern}
\usepackage{mathrsfs}
\usepackage{pifont}
\usepackage{array}
\usepackage[final]{microtype}
\usepackage{paragraphs}
\usepackage{enumitem}
\usepackage{amsmath}
\usepackage{amsfonts}
\usepackage{amssymb}
\usepackage{textcomp}
\usepackage{stmaryrd}
\usepackage{colonequals}
\usepackage{mathtools}
\usepackage{siunitx}
\usepackage{booktabs}
\usepackage{varioref}
\usepackage{tikz}
\usetikzlibrary{cd,babel}

\DeclareMathAlphabet{\mathbx}{U}{bbold}{m}{n}

\usepackage{tikz}
\usetikzlibrary{decorations.markings}
\usetikzlibrary{intersections}

\usepackage{framed}
\colorlet{shadecolor}{green!20}

\makeatletter
\AtBeginDocument{%
  \expandafter\let\csname[\endcsname\relax
  \expandafter\let\csname]\endcsname\relax
  \expandafter\DeclareRobustCommand\csname[\expandafter\endcsname\expandafter{%
    \csname begin\endcsname{equation}%
  }%
  \expandafter\DeclareRobustCommand\csname]\expandafter\endcsname\expandafter{%
    \csname end\endcsname{equation}%
  }%
}
\makeatother
\mathtoolsset{showonlyrefs,showmanualtags}

\usepackage[final]{hyperref}
\hypersetup{
  colorlinks,
  linkcolor=-red!75!green!50,
  citecolor=-red!75!green!50,
  urlcolor=green!30!black
}

\usepackage[
        initials, 
        msc-links]{amsrefs}

\usepackage{amsthm}
\usepackage{thmtools}
\declaretheoremstyle[
        headformat=\NUMBER.\hskip.5ex\NAME\NOTE,
        spaceabove=2ex,
        headfont=\bfseries,
        bodyfont=\itshape
        ]{mystyle}
\declaretheoremstyle[
        numbered=no,
        spaceabove=1ex,
        bodyfont=\itshape
        ]{mystyleempty}
\declaretheoremstyle[
        headformat={{\bfseries\NUMBER.}\ {\itshape\NAME}\NOTE},
        spaceabove=1ex,
        headpunct={.},
        headfont=\itshape,
        bodyfont=\normalfont
        ]{mystyleplain}
\declaretheoremstyle[
        spaceabove=1ex,
        headpunct={.},
        headfont=\bfseries,
        bodyfont=\itshape
        ]{mystylespecial}
\declaretheorem[sibling=para, style=mystyle]{Lemma}
\declaretheorem[style=mystyleempty, name=Lemma]{Lemma*}
\declaretheorem[sibling=para, style=mystyle]{Proposition}
\declaretheorem[style=mystyleempty, name=Proposition]{Proposition*}

\declaretheorem[style=mystyleempty, name=Corollary]{Corollary*}

\newcounter{theorem}

\declaretheorem[numberlike=theorem, name=Theorem]{Theorem}

\newlist{thmlist}{enumerate}{1}
\setlist[thmlist]{
        nolistsep,
        ref={\mdseries\textup{(\emph{\roman*})}},
        label={\mdseries\textup{(\emph{\roman*})}},
        before={\advance\mathindent\leftmargin}
        }
\newcommand\thmitem[1]{\textup{(\emph{\romannumeral #1})}}

\makeatletter
\newlength{\@thlabel@width}
\newcommand{\fixhspace}{%
        \settowidth{\@thlabel@width}{\hskip0.5em\itshape1.}%
        \sbox{\@labels}{%
                \unhbox\@labels%
                \hspace{\dimexpr-\leftmargin+\labelsep+\@thlabel@width-\itemindent}
                }
        }
\makeatother

\makeatletter
\renewenvironment{proof}[1][\proofname]{\par
  \pushQED{\qed}%
  \normalfont \topsep1ex%
  \trivlist
  \item[\hskip\labelsep
        \itshape
    #1\@addpunct{.}]\ignorespaces
}{%
  \popQED\endtrivlist\@endpefalse
}
\makeatother

\newlist{tfae}{enumerate}{1}
\setlist[tfae]{
        nolistsep,
        ref={\mdseries\textup{(\emph{\alph*})}},
        label={\mdseries\textup{(\emph{\alph*})}},
        before={\advance\mathindent\leftmargin}}
\makeatletter
\newcommand\tfaeitem[1]{{\textup{(\emph{\@alph #1})}}}

\makeatother

\newcommand\claim[2][.8]{%
  \begin{minipage}{#1\displaywidth}%
  \itshape
  #2
  \end{minipage}%
}

\usepackage{caption}
\makeatletter
\renewcommand\tableofcontents{%
  \begin{center}
  \begin{minipage}{.6\textwidth}
  \footnotesize
  \@starttoc{toc}
  \end{minipage}
  \end{center}
}
\makeatother

\DeclareFontFamily{OT1}{slmss}{}
\DeclareFontShape{OT1}{slmss}{m}{n}
     {<-8.5> s*[1.1] rm-lmss8
      <8.5-9.5> s*[1.1] rm-lmss9
      <9.5-11> s*[1.1] rm-lmss10
      <11-15.5> s*[1.1] rm-lmss12
      <15.5-> s*[1.1] rm-lmss17
     }{}
\DeclareSymbolFont{sfoperators}{OT1}{slmss}{m}{n}
\DeclareSymbolFontAlphabet{\mathsf}{sfoperators}
\makeatletter
\def\operator@font{\mathgroup\symsfoperators}
\makeatother

\newcommand\NN{\mathbb{N}}
\newcommand\ZZ{\mathbb{Z}}

\renewcommand\epsilon{\varepsilon}
\newcommand\D{\mathscr{D}}
\newcommand\id{\mathrm{id}}

\newcommand\U{\mathscr{U}}

\DeclarePairedDelimiter\lin{\langle}{\rangle}

\newcommand\something[1]{\fbox{$#1$}\,}
\DeclareMathOperator{\gr}{gr}
\DeclareMathOperator{\ad}{ad}

\DeclareMathOperator{\gldim}{gldim}
\DeclareMathOperator{\pdim}{pdim}

\DeclareMathOperator{\GL}{GL}

\DeclareMathOperator{\Ext}{Ext}

\DeclareMathOperator{\Der}{Der}
\DeclareMathOperator{\InnDer}{InnDer}
\DeclareMathOperator{\Aut}{Aut}
\DeclareMathOperator{\End}{End}
\DeclareMathOperator{\Exp}{Exp}

\newcommand\lmod[1]{{}_{#1}\mathsf{Mod}}
\newcommand\rmod[1]{\mathsf{Mod}_{#1}}

\let\*\bullet

\newcommand\N{\mathscr{N}}

\renewcommand\O{\mathscr{O}}

\newcommand\s{\mathsf{s}}

\newcommand\B{\mathbf{B}}
\renewcommand\AA{\mathbb{A}}

\newcommand\HH{H\!H}
\newcommand\HC{H\!C}
\renewcommand\H{\mathscr{H}}
\newcommand\kk{\Bbbk}
\newcommand\place{\mathord-}

\newcommand\X{\mathfrak{X}}
\newcommand\pt[1]{_{(#1)}}

\newcommand\A{\mathcal{A}}

\newcommand\F{\mathcal{F}}
\renewcommand\P{\mathbf{P}}
\let\oldS\S
\renewcommand\S{\mathscr{S}} 

\newcommand\xx{\hat{x}}
\newcommand\yy{\hat{y}}

\newcommand\EE{\hat{E}}
\newcommand\DD{\hat{D}}

\newcommand\w{\mathsf{w}}

\setitemize{nolistsep, listparindent=\parindent}

\newcommand\newterm[1]{\textit{#1}}

\usepackage{pifont}
\allowdisplaybreaks

\title[Algebras of differential operators]{Hochschild cohomology of
algebras of differential operators tangent to a central arrangement of
lines}

\author{Francisco Kordon}
\address{Departamento de Matem\'atica, Facultad de Ciencias Exactas y
Naturales, Universidad de Buenos Aires. Ciudad Universitaria,
Pabell\'on I (1428) Ciudad de Buenos Aires, Argentina.}
\email{fkordon@dm.uba.ar}

\author{Mariano Suárez-Álvarez}
\address{Departamento de Matem\'atica, Facultad de Ciencias Exactas y
Naturales, Universidad de Buenos Aires. Ciudad Universitaria,
Pabell\'on I (1428) Ciudad de Buenos Aires, Argentina.}
\email{mariano@dm.uba.ar}

\date{July 25, 2018}

\thanks{IMAS (CONICET) -- Universidad de Buenos Aires. This work has been
supported by the projects UBACYT 20020130100533BA, PIP-CONICET 112–201501–
00483CO, PICT 20150366 and MATHAMSUD-REPHOMOL}

\subjclass[2010]{Primary 16E40; Secondary 14N20}

\begin{document}

\begin{abstract}
Given a central arrangement of lines~$\A$ in a $2$-dimensional vector
space~$V$ over a field of characteristic zero, we study the algebra
$\D(\A)$ of differential operators on~$V$ which are logarithmic along~$\A$.
Among other things we determine the Hochschild cohomology of~$\D(\A)$ as a
Gerstenhaber algebra, establish a connection between that cohomology and the
de Rham cohomology of the complement~$M(\A)$ of the arrangement, determine
the isomorphism group of~$\D(\A)$ and classify the algebras of that form up
to isomorphism.
\end{abstract}

\maketitle

Let us fix a ground field~$\kk$ of characteristic zero, a vector space~$V$
and a central arrangement of hyperplanes~$\A$ in~$V$. We let $S$ be
the algebra of polynomial functions of~$V$, fix a defining polynomial~$Q\in
S$ for~$\A$, and consider, following K.~Saito~\cite{Saito}, the Lie algebra
  \[
  \Der(\A) = \{\delta\in\Der(S):\delta(Q)\in QS\}
  \]
of derivations of~$S$ logarithmic with respect to~$\A$, which is,
geometrically speaking, the Lie algebra of vector fields on~$V$ which are
tangent to the hyperplanes of~$\A$. This Lie algebra is a very interesting
invariant of the arrangement and has been the subject of a lot of work~---~we refer to the book of P.~Orlik and H.~Terao \cite{OT} and the one by
A.~Dimca~\cite{Dimca} for surveys on this subject. In particular,
using this Lie algebra we can define an important class of arrangements: we say
that an arrangement $\A$ is \newterm{free} if $\Der(\A)$ is free as a left
$S$-module. For example, central arrangements of lines in the plane are
free, as are, according to a beautiful result of Terao~\cite{Terao}, the
arrangements of reflecting hyperplanes of a finite group generated by
pseudo-reflection

Now, along with $\Der(\A)$ we can consider also the associative algebra~$\D(\A)$
generated inside the algebra $\End_\kk(S)$ of linear endomorphisms of the
vector space~$S$ by $\Der(\A)$ and the set of maps given by left
multiplication by elements of~$S$: we call it the \newterm{algebra of
differential operators tangent} to the arrangement~$\A$. When $\A$ is free,
this coincides with the algebra of differential operators on~$S$ which
preserve the ideal~$QS$ of~$S$ and all its powers, studied for example by
F.\,J.~Calderón-Moreno \cite{CM} or by the second author
in~\cite{SA:idealizer}.

The purpose of this paper is to study, from the point of view of
non-commutative algebra and homological algebra, this algebra~$\D(\A)$ in
the simplest case of free arrangement, that of central line arrangements.

\bigskip

Let us describe briefly our results. We thus assume in what follows that
$\A$ is a central arrangement of $r+2$ lines in a $2$-dimensional vector
space~$V$, and for simplicity we suppose that $\A$ has at least five lines,
so that $r\geq3$. We let $Q\in S$ be a defining polynomial for~$\A$, that is,
a square-free product of linear forms on~$V$ with the union of the
hyperplanes of~$\A$ as zero locus. As $S$ is a subalgebra of~$\D(A)$, we
view $Q$ as an element of the latter.

\begin{Theorem}
The algebra $\D(\A)$ is a noetherian domain, it has global dimension~$4$
and projective dimension as a bimodule over itself also equal to~$4$. The
Hochschild cohomology~$\HH^\bullet(\D(\A))$ of~$\D(\A)$ has Hilbert series
  \[
  \sum_{i\geq0} \dim\HH^i(\D(\A))\cdot t^i
        = 1 + (r+2)t + (2r+3)t^2 + (r+2)t^3.
  \]
The algebra~$\D(\A)$ has Hochschild homology and cyclic homology
isomorphic to those of a polynomial algebra $\kk[X]$, and periodic homology and
higher $K$-theory isomorphic to that of the ground field~$\kk$. It is a twisted
Calabi--Yau algebra of dimension~$4$, the element $Q$ of~$\D(A)$ is normal, 
and the modular automorphism $\sigma:\D(\A)\to\D(\A)$ of~$\D(\A)$ is the
unique one such that for all $a\in\D(\A)$ one has
  \[
  Qa=\sigma(a)Qa.
  \]
\end{Theorem}

These claims are contained in Propositions~\pref{prop:hh}, \pref{prop:hom},
\pref{prop:CY} and~\pref{prop:conn}. In Propositions~\pref{prop:cup}
and~\pref{prop:bracket} we describe completely the cup product and the
Gerstenhaber Lie structure on~$\HH^\bullet(\D(\A))$ --- we refer to their
statements for the precise details, which are technical. The
calculations needed in order to do these computations are annoyingly involved.

We obtain a very concrete description of $\HH^1(\D(\A))$ in
Proposition~\pref{prop:partial}:

\begin{Theorem}
Let $Q=\alpha_1\cdots\alpha_{r+2}$ be a factorization of the defining
polynomial as a product of linear factors, so that
$\alpha_1$,~\dots,~$\alpha_{r+2}$ are linear polynomials on~$V$ whose zero
loci are the hyperplanes of~$\A$.
\begin{thmlist}

\item For each $i\in\{1,\dots,r+2\}$ there is a unique derivation
$\partial_i:\D(\A)\to\D(\A)$ such that $\partial_i(f)=0$ for all $f\in S$
and $\partial_i(\delta)=\delta(\alpha_i)/\alpha_i$ for all
$\delta\in\Der(\A)$.

\item The set of classes of $\partial_1$,~\dots,~$\partial_{r+2}$
in~$\HH^1(\D(\A))$, which we view as the space of outer derivations of the
algebra~$\D(\A))$, is a basis.
\end{thmlist}
\end{Theorem}

The elements $\partial_1$,~\dots,~$\partial_{r+2}$ are canonically
determined and in a natural bijection with the set of hyperplanes. We do
not have a description along the same lines of the rest of the cohomology. 
In Proposition~\pref{prop:partial:sub}, though,  we do obtain the following piece of
information:

\begin{Theorem}
The subalgebra~$\H$ of~$\HH^\bullet(\D(\A))$ generated by the
component~$\HH^1(\D(\A))$ of degree~$1$ is isomorphic to the de Rham
cohomology of the complement $M(\A)$ of the arrangement. It is freely
generated as a graded-commutative algebra by the $r+2$ elements
$\partial_1$,~\dots,~$\partial_{r+2}$ of~$\HH^1(\D(\A))$ subject to the
relations
  \[
  \partial_i\smile\partial_j 
        + \partial_j\smile\partial_k
        + \partial_k\smile\partial_i
        = 0,
  \]
one for each choice of three pairwise distinct elements $i$,~$j$,~$k$
of~$\{1,\dots,r+2\}$. 
\end{Theorem}

Using our precise description of~$\HH^1(\D(\A))$ and the techniques of
J.~Alev and M.~Chamarie \cite{AC}, we arrive in Section~\pref{sect:autos}
at a description of the
automorphism group of the algebra~$\D(\A)$. Since the arrangement~$\A$ is
central, the Lie algebra $\Der(\A)$ is a \emph{graded} $S$-module, and that
grading turns $\D(\A)$ into a graded algebra: we will use this structure
in the following result.

\begin{Theorem}
Let $G$ be the subgroup of~$\GL(V)$ of maps which preserve the
arrangement~$\A$. 
\begin{thmlist}

\item There is an action of~$G$ on a vector space~$W$ of
dimension~$r+2$ such that the semidirect product $G\ltimes W$ is isomorphic
to the group $\Aut_0(\D(\A))$ of algebra automorphisms of~$\D(\A)$ which
respect the grading.

\item An element of~$\D(\A)$ is locally $\ad$-nilpotent if and only if it
belongs to~$S$. The set $\Exp(\A)=\{\exp\ad(f):f\in S\}$ of the
automorphisms of~$\D(\A)$ obtained as exponentials of locally
$\ad$-nilpotent elements is a subgroup of the full group of
automorphisms~$\Aut(\D(\A))$.

\item There is an action of~$\Aut_0(\D(\A))$ on~$\Exp(\A)$ such that there
is an isomorphism of groups
$\Aut(\D(\A)) = \Aut_0(\D(\A))\ltimes\Exp(\A)$.

\end{thmlist}
\end{Theorem}

\smallskip

This knowledge of the automorphism group of~$\D(\A))$ allows us to describe
the set of normal elements of the algebra and its birational class:

\begin{Theorem}
The set of normal elements of~$\D(\A)$ is the saturated multiplicatively
closed subset generated by~$Q$. The maximal normal localization
of~$\D(\A)$, which is therefore $\D(\A)[\tfrac1Q]$, is isomorphic to the
localization of the Weyl algebra $\mathit{Diff}(S)[\tfrac1Q]$.
\end{Theorem}

Finally, using ---as it is often done--- normal elements, we are able to
classify the algebras under study up to isomorphism:

\begin{Theorem}
Let $\A$ and $\A'$ be two central arrangements of lines in~$V$. The
algebras $\D(\A)$ and~$\D(\A')$ are isomorphic if and only if the
arrangements $\A$ and~$\A'$ themselves are linearly isomorphic.
\end{Theorem}

This means, essentially, that we can reconstruct the arrangement from the
algebra~$\D(\A)$ of its differential operators.

\bigskip

We expect most of the above results to hold in the general case of a free
arrangement of hyperplanes of arbitrary rank. As our computations here make
clear, some technology is needed in order to deal with more complicated
cases. In future work, we will show how to organize this computation using
the language of Lie--Rinehart pairs \cite{Rinehart} and their cohomology theory.
On the other hand, one can interpret the second cohomology space
$\HH^2(\D(\A))$ as classifying infinitesimal deformations of
the algebra~$\D(\A)$ and use $\HH^3(\D(\A))$ and our description of the
Gerstenhaber bracket to study the deformation theory of~$\D(\A)$. This
produces a somewhat concrete interpretation of the second cohomology space
in geometrical terms. As this involves quite a bit of calculation, we defer
the exposition of these results to a later paper.

The contents of this paper are part of the doctoral thesis of the first
author.

\bigskip

The paper is organized as follows. We start in Section~\pref{sect:algebra}
by giving a concrete realization of the algebra~$\D(\A)$ as an iterated Ore
extension of a polynomial ring and proving some useful lemmas. In
Section~\pref{sect:res} we construct a resolution for~$\D(\A)$ and in
Sections~\pref{sect:hh} and~\pref{sect:cup} we present the
computation of the Hochschild cohomology $\HH^\bullet(\D(\A))$ and
its Gerstenhaber algebra structure. Section~\pref{sect:rest} gives the much
easier determination of the Hochschild homology, cyclic homology, periodic
homology and $K$-theory of our algebra, followed by the proof, in
Section~\pref{sect:cy}, of the twisted Calabi--Yau property. Finally, in
the last section we determine the automorphism group of $\D(\A)$ and
classify the algebras of this form up to isomorphism.

\smallskip

\noindent\textbf{Some notations} We will use the symbols $\triangleright$
and $\triangleleft$ to denote the left and right actions of an algebra on a
bimodule whenever this improves clarity. We will have a ground field~$\kk$
of characteristic zero. All vector spaces and algebras are implicitly
defined over~$\kk$, and unadorned $\otimes$ and $\hom$ are taken with
respect to~$\kk$. If $M$ is a vector space, we will often denote by
  \[
  \something{M}
  \]
an element of $M$ about which we do not need to be specific.

We refer to the book \cite{OT} for a general reference about hyperplane
arrangements and their derivations, and to C.~Weibel's book \cite{Weibel}
for generalities about homological algebra and, in particular, Hochschild,
cyclic and periodic theories.

\section{The algebra of differential operators associated to a central
arrangement of lines}
\label{sect:algebra}

\paragraph We fix once and for all a ground field~$\kk$ of characteristic
zero and put $S=\kk[x,y]$. We view $S$ as a graded algebra as usual, with
both~$x$ and~$y$ of degree~$1$, and for each $p\geq0$ we write $S_p$ the
homogeneous component of~$S$ of degree~$p$.

We write $\Der(S)$ the Lie algebra of derivations of~$S$, which is a free left
graded $S$-module, freely generated by the usual partial derivatives
$\partial_x$,~$\partial_y:S\to S$, which are homogeneous elements
of~$\Der(S)$ of degree~$-1$. On the other hand, we write
$\D(S)$ the associative algebra of regular differential operators on~$S$,
as defined, for example, in~\cite{MR}*{\oldS15.5}.
As this is by definition a subalgebra of~$\End_\kk(S)$, there is a tautological
structure of left $\D(S)$-module on~$S$.

There is an injective morphism of algebras $\phi:S\to\D(S)$ such that
$\phi(s)(a)=as$ for all $s$,~$a\in S$ which we will view as an
identification; elements in its image are the differential operators of
order zero. Since $S$~is a regular algebra, the algebra~$\D(S)$ is
generated as a subalgebra of~$\End_\kk(S)$  by~$S$ and~$\Der(S)$;
see~\cite{MR}*{Corollary~15.5.6}. A consequence of this is that $\D(S)$ is
generated as an algebra by $x$,~$y$,~$\partial_x$ and~$\partial_y$, and in
fact these elements generate it freely subject to the relations
  \begin{align}
  &  [x,y]=[\partial_x,y]=[\partial_y,x]=[\partial_x,\partial_y]=0,
  && [\partial_x,x]=[\partial_y,y]=1.
  \end{align}
It follows easily from this that $\D(S)$ has a $\ZZ$-grading with $x$ and $y$ in
degree~$1$ and $\partial_x$ and~$\partial_y$ in degree~$-1$, and that with respect
to this grading, $S$ is a graded $\D(S)$-module.

\paragraph\label{p:start} We fix an integer $r\geq-1$ and consider a central
arrangement $\A$ of $r+2$ lines in the plane~$\AA^2$. Up to a change of
coordinates, we may assume that the line with equation $x=0$ is one of the
lines in~$\A$, so that the defining polynomial~$Q$ of the arrangement is of the
form $xF$ for some square-free homogeneous polynomial $F\in S$ of degree
$r+1$ which does not have $x$ as a factor. Up to multiplying by a scalar,
which does not change anything substantial, we may assume that
  \(
  F=x\bar F+y^{r+1}
  \)
for some $\bar F\in S_r$.

We let $\Der(\A)$ be the Lie
algebra of derivations of~$S$ that preserve the arrangement, as
in~\cite{OT}*{\oldS4.1}, so that
  \[
  \Der(\A)=\{d\in\Der(S):d(Q)\in QS\}.
  \]
This a graded Lie subalgebra of~$\Der(S)$.
The two derivations
  \begin{align}
  &  E=x\partial_x+y\partial_y,
  && D=F\partial_y
  \end{align}
are elements of~$\Der(\A)$ of degrees~$0$ and~$r$, and it follows
immediately from Saito's criterion \cite{OT}*{Theorem 4.19} that the set
$\{E,D\}$ is a basis of $\Der(\A)$ as a graded $S$-module; this is the
content of Example 4.20 in that book.

The \newterm{algebra of differential operators tangent to the
arrangement~$\A$} is the subalgebra~$\D(\A)$ of~$\D(S)$ generated by~$S$
and~$\Der(\A)$. It follows immediately from the remarks above that $\D(\A)$
is generated by $x$,~$y$,~$E$ and~$D$, and a computation shows that the
following commutation relations hold in~$\D(\A)$:
  \begin{align}
  &  [y,x] = 0, \\
  &  [D,x] = 0,
  && [D,y] = F, \label{eq:relations}\\
  &  [E,x] = x,
  && [E,y] = y,
  && [E,D] = rD.
  \end{align}
Since these generators are homogeneous elements in~$\D(S)$ ---with $E$ of
degree~$0$, $x$ and~$y$ of degree~$1$ and $D$ of degree~$r$--- we see that
the algebra $\D(\A)$ is a graded subalgebra of~$\D(S)$ and, by restricting
the structure from~$\D(S)$, that $S$ is a graded $\D(\A)$-module.

The set of commutation relations given above is in fact a
presentation of the algebra~$\D(\A)$. More precisely, we have:

\begin{Lemma*}\plabel{lemma:ore}
The algebra~$\D(\A)$ is isomorphic to the iterated Ore extension
$S[D][E]$. It is a noetherian domain and the set
$\{x^iy^jD^kE^l:i,j,k,l\geq0\}$ is a $\kk$-basis for~$\D(\A)$.
\end{Lemma*}

Here we view $D$ as a derivation of~$S$, so that we way
construct the Ore extension $S[D]$, and view $E$ as a derivation of this
last algebra, so as to be able extend once more to obtain~$S[D][E]$.

\begin{proof}
It is clear at this point that the obvious map
$\pi:S[D][E]\to\D(\A)$ is a surjective morphism of algebras, so we
need only prove that it is injective. To do that, let us suppose that there
exists a non-zero element $L$ in~$S[D][E]$ whose image under the
map~$\pi$ is zero, and suppose that $L=\sum_{i,j\geq0}f_{i,j}D^iE^j$, with
coefficients~$f_{i,j}\in S$ for all $i$,~$j\geq0$, almost all of which are
zero. As $L$ is non-zero, we may consider the number
$m=\max\{i+j:f_{i,j}\neq0\}$.

Let us now fix a point $p=(a,b)\in\AA^2$ which is not on any line of the
arrangement~$\A$, so that $aF(a,b)\neq0$, and let $\O_p$ be the completion
of $S$ at the ideal $(x-a,y-b)$ or, more concretely, the algebra of formal
series in $x-a$ and $y-b$. We view $\O_p$ as a left module over $\D(S)$ in
the tautological way and, by restriction, as a left $\D(\A)$-module.
There exist formal series $\phi$ and~$\psi$ in~$\O_p$ such that
  \begin{align}
  &  E\cdot\phi = 1,
  && D\cdot\phi = 0,
  && E\cdot\psi = 0,
  && D\cdot\psi = x^r.
\intertext{Indeed, we may choose $\phi=\ln x$ to satisfy the first two
conditions, and the last two ones are equivalent to the equations}
  &  \partial_x\psi = -\frac{x^{r-1}y}{F},
  && \partial_y\psi = \frac{x^r}{F},
  \end{align}
which can be solved for~$\psi$, as the usual well-known sufficient
integrability condition from elementary calculus holds. If now
$s$,~$t\in\NN_0$ are such that $s+t=m$, a straightforward computation shows
that
  \(
  L\cdot\phi^s\psi^t=s!t!x^{rt}f_{s,t}
  \)
in~$\O_p$, and this implies that $f_{s,t}=0$. This contradicts the choice
of~$m$ and this contradiction proves what we want.
\end{proof}

\paragraph We will use the following two simple lemmas a few
times:

\begin{Lemma*}\plabel{lemma:indep}
Suppose that $r\geq2$.
If $\alpha$,~$\beta\in S_1$ are such that $\alpha F_x+\beta F_y=0$,
then $\alpha=\beta=0$.
\end{Lemma*}

The conclusion of this statement is false if $r<2$.

\begin{proof}
Suppose that $F_1$, $F_2$ and $F_3$ are three distinct linear
factors of~$F$ (here is where we need the hypothesis that $r$ is at least
$2$) so that $F=F_1F_2F_3F'$ for some $F'\in S_{r-2}$; as
$F$ has degree at least $3$, this is possible. We have
$F_x\equiv F_{1x}F_2F_3F'$ and $F_y\equiv F_{1y}F_2F_3F'$ modulo~$F_1$, so
that $(\alpha F_{1x}+\beta F_{1y})F_2F_3F'\equiv0\mod F_1$. Since $F$ is
square free, this tells us that $F_1$ divides $\alpha F_{1x}+\beta F_{1y}$
and, since both polynomials have the same degree and $F_1\neq0$, that there exists a
scalar $\lambda$ such that $\alpha F_{1x}+\beta F_{1y}=\lambda F_1$. Of
course, we can do the same with the other two factors~$F_2$ and~$F_3$. We
can state this by saying that the matrix
$\begin{psmallmatrix}\alpha_x&\beta_x\\ \alpha_y&\beta_y\end{psmallmatrix}$
has the three vectors
$\begin{psmallmatrix}F_{1x}\\F_{1y}\end{psmallmatrix}$,
$\begin{psmallmatrix}F_{2x}\\F_{2y}\end{psmallmatrix}$ and
$\begin{psmallmatrix}F_{3x}\\F_{3y}\end{psmallmatrix}$ as eigenvectors.
Since no two of these are linearly dependent, because $F$ is square-free,
this implies that the matrix
is in fact a scalar multiple of the identity, and there is a $\mu\in\kk$
such that $\alpha=\mu x$ and $\beta=\mu y$. The hypothesis is then that
$\mu (r+1)F=\mu(xF_x+yF_y)=0$, so that $\mu=0$. This proves the claim.
\end{proof}

\begin{Lemma}
If $\alpha_1$,~\dots,~$\alpha_{r+1}\in S_1$ are such that
$F=\prod_{i=1}^{r+1}\alpha_i$, then the set of quotients
  \(
  \{\tfrac{F}{\alpha_1},\dots,\tfrac{F}{\alpha_{r+1}}\}
  \)
is a basis for~$S_r$.

\end{Lemma}

\begin{proof}
Suppose $c_1$,~\dots,~$c_{r+1}\in\kk$ are scalars such that
$\sum_{i=1}^{r+1}c_i\frac{F}{\alpha_i}=0$. If $j\in\{1,\dots,r+1\}$, we
then have $c_j\frac{F}{\alpha_j}\equiv0$ modulo~$\alpha_j$ and, since
$F$ is square-free, this implies that in~fact $c_j=0$. The set
$\{\tfrac{F}{\alpha_1},\dots,\tfrac{F}{\alpha_{r+1}}\}$ is therefore linearly
independent. Since $\dim S_r=r+1$, this completes the proof.
\end{proof}

\section{A projective resolution}
\label{sect:res}

\paragraph We keep the situation of the previous section, and write from
now on~$A$ instead of~$\D(\A)$. Our immediate objective is to construct a
projective resolution of~$A$ as an $A$\nobreakdash-bimodule, and we do this by looking
at~$A$ as a deformation of a commutative polynomial algebra, which suggests
that it should have a resolution resembling the usual Koszul complex.

\paragraph If $U$ is a vector space and~$u\in U$, there are
derivations $\nabla_x^u$,~$\nabla_y^u:S\to S\otimes U\otimes S$ of~$S$ into
the $S$-bimodule $S\otimes U\otimes S$ uniquely determined by the condition
that
  \begin{align}
  &   \nabla_x^u(x) = 1\otimes u\otimes 1,
  &&  \nabla_x^u(y) = 0,
  &&  \nabla_y^u(x) = 0,
  &&  \nabla_y^u(y) = 1\otimes u\otimes 1,
  \end{align}
and in fact we have, for every $i$,~$j\geq0$, that
  \begin{align}
  &  \nabla_x^u(x^iy^j) = \sum_{s+t+1=i}x^s\otimes u\otimes x^ty^j,
  && \nabla_y^u(x^iy^j) = \sum_{s+t+1=j}x^iy^s\otimes u\otimes y^s.
  \end{align}

We consider the derivation $\nabla=\nabla_x^x+\nabla_y^y:S\to S\otimes
S_1\otimes S$; it is the unique derivation such that
$\nabla(\alpha)=1\otimes\alpha\otimes 1$ for all $\alpha\in S_1$.
There is, on the other hand, a unique morphism of $S$-bimodules $d:S\otimes
S_1\otimes S\to S\otimes S$ such that
$d(1\otimes\alpha\otimes1)=\alpha\otimes1-1\otimes\alpha$ for all
$\alpha\in S_1$, and we have
  \[
  d(\nabla(f)) = f\otimes1-1\otimes f
  \]
for all $f\in S$. To check this last equality, it is enough to notice that
$d\circ\nabla:S\to S\otimes S$ is a derivation and, since $S_1$ generates
$S$ as an algebra, that the equality holds when $f\in S_1$.

\paragraph Let $V$ be the subspace of~$A$ spanned by $x$,~$y$,~$D$ and~$E$.
This is a graded subspace and its grading induces on the
exterior algebra~$\Lambda^\bullet(V)$ an internal grading.
If $\omega$ is an element of an exterior power $\Lambda^p(V)$ of~$V$, we write
$(\place)\wedge\omega$ the map of $A$-bimodules
  \[
  A\otimes S_1\otimes A\to A\otimes\Lambda^{p+1}V\otimes A
  \]
such that $(1\otimes\alpha\otimes 1)\wedge\omega=1\otimes
\alpha\wedge\omega\otimes 1$ for all $\alpha\in S_1$.

\paragraph\label{p:P} There is a chain complex~$\P$ of free graded
$A$-bimodules of the form
  \[ \label{eq:res}
  \begin{tikzcd}[column sep=1.75em]
  A|\Lambda^4V|A \arrow[r, "d_4"]
        & A|\Lambda^3V|A \arrow[r, "d_3"]
        & A|\Lambda^2V|A \arrow[r, "d_2"]
        & A|V|A \arrow[r, "d_1"]
        & A|A
  \end{tikzcd}
  \]
with $A^e$-linear maps homogeneous of degree zero and such that
\def\tmp{\color{blue}}
  \begin{gather*}
  d_1(1|v|1) = [v,1|1], \qquad\forall v\in V; \\[3pt]
  d_2(1|x\wedge y|1) = [x,1|y|1]-[y,1|x|1]; \\
  d_2(1|x\wedge E|1) = [x,1|E|1]-[E,1|x|1] + 1|x|1; \\
  d_2(1|y\wedge E|1) = [y,1|E|1]-[E,1|y|1] + 1|y|1; \\
  d_2(1|x\wedge D|1) = [x,1|D|1]-[D,1|x|1]; \\
  d_2(1|y\wedge D|1) = [y,1|D|1]-[D,1|y|1] + \nabla(F); \\
  d_2(1|D\wedge E|1) = [D,1|E|1]-[E,1|D|1] + r|D|1; \\[3pt]
  d_3(1|x\wedge y\wedge D|1)
        =  [x,1|y\wedge D|1] - [y,1|x\wedge D|1] + [D,1|x\wedge y|1] + \nabla(F)\wedge x;\\
  d_3(1|x\wedge y\wedge E|1)
        =  [x,1|y\wedge E|1] - [y,1|x\wedge E|1] + [E,1|x\wedge y|1] - 2|x\wedge y|1;  \\
  \!\begin{multlined}[0.85\displaywidth]
  d_3(1|x\wedge D\wedge E|1)
        =  [x,1|D\wedge E|1] - [D,1|x\wedge E|1] + [E,1|x\wedge D|1] \\
        - (r+1)|x\wedge D|1;
  \end{multlined}
        \\
  \!\begin{multlined}[0.9\displaywidth]
  d_3(1|y\wedge D\wedge E|1)
        =  [y,1|D\wedge E|1] - [D,1|y\wedge E|1] + [E,1|y\wedge D|1] \\
                + \nabla(F)\wedge E - (r+1)|y\wedge D|1;
  \end{multlined}
                        \\[3pt]
  d_4(1|x\wedge y\wedge D\wedge E|1) =
  \!\begin{multlined}[0.625\displaywidth][t]
  [x,1|y\wedge D\wedge E|1] - [y,1|x\wedge D\wedge E|1]  \\[3pt]
        + [D,1|x\wedge y\wedge E|1] - [E,1|x\wedge y\wedge D|1] \\[0pt]
        + \nabla(F)\wedge x\wedge E + (r+2)|x\wedge y\wedge D|1.
  \end{multlined}
  \end{gather*}
That $\P$ is indeed a complex follows from a direct calculation. More
interestingly, it is exact:

\begin{Lemma*}
The complex~$\P$ is a projective resolution of~$A$ as an $A$-bimodule, with
augmentation $d_0:A|A\to A$ such that $d_0(1|1)=1$.
\end{Lemma*}

\begin{proof}
For each $p\in\NN_0$ we consider the subspace
  \(
  \F_pA = \lin{x^iy^jD^kE^l:k+l\leq p}
  \)
of~$A$. As a consequence of Lemma~\ref{lemma:ore},
one sees that $\F A=(\F_pA)_{p\geq0}$ is an exhaustive and increasing algebra
filtration on~$A$ and that the corresponding associated graded algebra~$\gr(A)$ is
isomorphic to the usual commutative polynomial ring $\kk[x,y,D,E]$. Since
$V$ is a subspace of~$A$, we can restrict the filtration of~$A$ to one
on~$V$, and the latter induces as usual a filtration on each exterior
power~$\Lambda^pV$. In this way we obtain a filtration on each component
of the complex~$\P$, which turns out to be  compatible with its
differentials, as can be checked by inspection. The complex~$\gr(\P)$
obtained from~$\P$ by passing to associated graded objects in each degree
is isomorphic to the Koszul resolution of~$\gr(A)$ as a $\gr(A)$-bimodule
and it is therefore acyclic over~$\gr(A)$. A standard argument using the
filtration of~$\P$ concludes from this that the complex~$\P$ itself acyclic over~$A$.
As its components are manifestly free $A$-bimodules, this proves the lemma.
\end{proof}

\paragraph One almost immediate application of having a bimodule projective
resolution for our algebra is in computing its global dimension.

\begin{Proposition*}\label{prop:gldim}
The global dimension of~$A$ is equal to~$4$.
\end{Proposition*}

Of course, as $A$ is noetherian, there is no need to distinguish between the left and
the right global dimensions.

\begin{proof}
If $\lambda\in\kk$ let $M_\lambda$ be the left $A$-module which as a vector
space is freely spanned by an element~$u_\lambda$ and on which the action
of~$A$ is such that $x\cdot u_\lambda=y\cdot u_\lambda=D\cdot u_\lambda=0$
and $E\cdot u_\lambda=\lambda u_\lambda$. It is easy to see that all
$1$-dimensional $A$-modules are of this form and that $M_\lambda\cong
M_\mu$ iff $\lambda=\mu$, but we will not need this.

The complex $\P\otimes_AM_\lambda$ is a projective resolution
of~$M_\lambda$ as a left $A$-module, and therefore the cohomology of
$\hom_A(\P\otimes_AM_\lambda,M_\mu)$ is canonically isomorphic to
$\Ext_A^\bullet(M_\lambda,M_\mu)$. Identifying as usual
$\hom_A(\P\otimes_AM_\lambda,M_\mu)$ to $M_\mu\otimes\Lambda^\bullet
V^*\otimes M_\lambda^*$, we compute that the complex is
  \begin{multline}
  \begin{tikzcd}[ampersand replacement=\&, cramped]
  M_\mu\otimes M_\lambda^*
        \arrow[r, "\delta^0"]
    \& M_\mu\otimes V^*\otimes M_\lambda^*
        \arrow[r, "\delta^1"]
    \& M_\mu\otimes \Lambda^2V^*\otimes M_\lambda^*
        \arrow[r, "\delta^2"]
    \& {}
 \end{tikzcd}
 \\
 \begin{tikzcd}[ampersand replacement=\&, cramped]
 {}
        \arrow[r]
    \& M_\mu\otimes \Lambda^3V^*\otimes M_\lambda^*
        \arrow[r, "\delta^3"]
    \& M_\mu\otimes \Lambda^4V^*\otimes M_\lambda^*
  \end{tikzcd}
  \end{multline}
with differentials given by
  \begin{align*}
  \MoveEqLeft
  \delta^0(1)
        = (\mu-\lambda)\otimes\hat E, \\
  \MoveEqLeft
  \delta^1(a\otimes\hat x+b\otimes\hat y+c\otimes\hat D+d\otimes\hat D) \\
       &= (\lambda+1-\mu)a\otimes\hat x\wedge\hat E
          +(\lambda+1-\mu)b\otimes\hat y\wedge\hat E
          +(\lambda+r-\mu)c\otimes\hat D\wedge\hat E, 
  \\
  \MoveEqLeft
  \delta^2(a\otimes\hat x\wedge\hat y+b\otimes\hat x\wedge\hat E
        +c\otimes\hat y\wedge\hat E+d\otimes\hat x\wedge\hat D
        +e\hat y\wedge\hat D+f\hat D\wedge\hat E) \\
       &= \begin{multlined}[.8\displaywidth][t]
          (\mu-\lambda-2)a\otimes\hat x\wedge\hat y\wedge\hat E
          + (\mu-\lambda-r-1)d\otimes\hat x\wedge\hat D\wedge\hat E \\
          + (\mu-\lambda-r-1)e\otimes\hat y\wedge\hat D\wedge\hat E,
          \end{multlined}
  \\
  \MoveEqLeft
  \delta^3(a\otimes\hat x\wedge\hat y\wedge\hat D
        +b\otimes\hat x\wedge\hat y\wedge\hat E
        +c\otimes\hat x\wedge\hat D\wedge\hat E
        +d\otimes\hat y\wedge\hat D\wedge\hat E) \\
      &= (\lambda+r+2-\mu)a\otimes\hat x\wedge\hat y\wedge\hat D\wedge\hat E.
  \end{align*}
An easy computation shows that
  \[
  \dim\Ext^p_A(M_\lambda,M_{\lambda+r+2}) =
    \begin{cases*}
    1, & if $p=3$ or $p=4$; \\
    0, & in any other case.
    \end{cases*}
  \]
In particular, $\Ext_A^4(M_\lambda,M_{\lambda+r+2})\neq0$ and therefore
$\gldim A\geq4$. On the other hand, we have constructed a projective
resolution of~$A$ as an $A$-bimodule of length~$4$, so that the projective
dimension of~$A$ as a bimodule is
$\pdim_{A^e}A\leq 4$. Since $\gldim A\leq\pdim_{A^e}A$, the proposition
follows from this.
\end{proof}

\section{The Hochschild cohomology of \texorpdfstring{$\D(\A)$}{D(A)}}
\label{sect:hh}

\paragraph\plabel{p:hhdifferentials} We want to compute the Hochschild
cohomology of the algebra~$A$. Applying the functor $\hom_{A^e}(\place,A)$
to the resolution~$\P$ of~\pref{p:P} we get, after standard
identifications, the cochain complex
  \[
  \begin{tikzcd}
  A \arrow[r, "d^0"]
        & A\otimes V^* \arrow[r, "d^1"]
                \arrow[l, shift left=1ex, dashed, "s^1"]
        & A\otimes \Lambda^2V^* \arrow[r, "d^1"]
                \arrow[l, shift left=1ex, dashed, "s^2"]
        & A\otimes \Lambda^3V^* \arrow[r, "d^2"]
                \arrow[l, shift left=1ex, dashed, "s^3"]
        & A\otimes \Lambda^4V^* \arrow[r]
                \arrow[l, shift left=1ex, dashed, "s^4"]
        & 0
  \end{tikzcd}
  \]
which we denote simply by $A\otimes\Lambda V^*$, with differentials such that
  \begin{gather*}
  d^0(a)
              = [x,a]\otimes \hat x
                +[y,a]\otimes \hat y
                +[D,a]\otimes \hat D
                +[E,a]\otimes \hat E;
                \\[5pt]
  \!\begin{multlined}[0.9\displaywidth]
  d^1(a\otimes \hat x)
              = -[y,a]\otimes \hat x\wedge\hat y
                +(a-[E,a])\otimes \hat x\wedge\hat E
                -[D,a]\otimes \hat x\wedge\hat D
                \\
                +\nabla_x^a(F)\otimes \hat y\wedge\hat D;
  \end{multlined}                \\
  d^1(a\otimes \hat y)
             = [x,a]\otimes \hat x\wedge\hat y
               +(a-[E,a])\otimes \hat y\wedge\hat E
               +(\nabla_y^a(F)-[D,a])\otimes \hat y\wedge\hat D;
               \\
  d^1(a\otimes \hat D)
            = [x,a]\otimes \hat x\wedge\hat D
              +[y,a]\otimes \hat y\wedge\hat D
              +(ra-[E,a])\otimes \hat D\wedge\hat E;
              \\
  d^1(a\otimes \hat E)
            = [x,a]\otimes \hat x\wedge\hat E
              + [y,a]\otimes \hat y\wedge\hat E
              + [D,a]\otimes \hat D\wedge\hat E;
              \\[5pt]
  d^2(a\otimes \hat x\wedge\hat y)
               = ([D,a] - \nabla_y^a(F))\otimes \hat x\wedge\hat y\wedge\hat D
                 + ([E,a]-2a)\otimes \hat x\wedge\hat y\wedge\hat E;
                \\
  d^2(a\otimes \hat x\wedge\hat E)
              = -[y,a]\otimes \hat x\wedge\hat y\wedge\hat E
                -[D,a]\otimes \hat x\wedge\hat D\wedge\hat E
                +\nabla_x^a(F)\otimes \hat y\wedge\hat D\wedge\hat E;
                \\
  d^2(a\otimes \hat y\wedge\hat E)
             = [x,a]\otimes \hat x\wedge\hat y\wedge\hat E
               +(\nabla_y^a(F)-[D,a])\otimes \hat y\wedge\hat D\wedge\hat E;
                \\
  d^2(a\otimes \hat x\wedge\hat D)
             = -[y,a]\otimes \hat x\wedge\hat y\wedge\hat D
               +([E,a]-(r+1)a)\otimes \hat x\wedge\hat D\wedge\hat E;
                \\
  d^2(a\otimes \hat y\wedge\hat D)
             = [x,a]\otimes \hat x\wedge\hat y\wedge\hat D
               +([E,a]-(r+1)a)\otimes \hat y\wedge\hat D\wedge\hat E;
                \\
  d^2(a\otimes \hat D\wedge\hat E)
             = [x,a]\otimes \hat x\wedge\hat D\wedge\hat E
               + [y,a]\otimes \hat y\wedge\hat D\wedge\hat E;
                \\[5pt]
  d^3(a\otimes \hat x\wedge\hat y\wedge\hat D)
             = (-[E,a]+(r+2)a)\otimes \hat x\wedge\hat y\wedge\hat D\wedge\hat E;
                \\
  d^3(a\otimes \hat x\wedge\hat y\wedge\hat E)
             = ([D,a]-\nabla_y^a(F))\otimes \hat x\wedge\hat y\wedge\hat D\wedge\hat E;
                \\
  d^3(a\otimes \hat x\wedge\hat D\wedge\hat E)
             = -[y,a]\otimes \hat x\wedge\hat y\wedge\hat D\wedge\hat E;
                \\
  d^3(a\otimes \hat y\wedge\hat D\wedge\hat E)
             = [x,a]\otimes \hat x\wedge\hat y\wedge\hat D\wedge\hat E.
  \end{gather*}
These differentials are homogeneous with respect to the natural
internal grading on the complex~$A\otimes\Lambda V^*$ coming from the
grading of $A$. We denote
$\gamma: A\otimes\Lambda V^*\to A\otimes\Lambda V^*$ the $\kk$-linear map whose
restriction to each homogeneous component
of the complex~$A\otimes\Lambda V^*$ is simply the multiplication by the
degree. There is a homotopy, drawn in
the diagram~\eqref{eq:res} with dashed arrows, with
  \begin{gather*}
  s^1(a\otimes\hat x+b\otimes\hat y+c\otimes\hat D+d\otimes\hat E)
        = d, \\
  \!\begin{multlined}[0.9\displaywidth]
  s^2(a\otimes \hat x\wedge\hat y
      + b\otimes \hat x\wedge\hat E
      + c\otimes \hat y\wedge\hat E
      + d\otimes \hat x\wedge\hat D
      + e\otimes \hat y\wedge\hat D
      + f\otimes \hat D\wedge\hat E) \\
        = -b\otimes\hat x-c\otimes\hat y-f\otimes\hat D,
  \end{multlined} \\
  \!\begin{multlined}[0.9\displaywidth]
  s^3(a\otimes \hat x\wedge\hat y\wedge\hat D
      + b\otimes \hat x\wedge\hat y\wedge\hat E
      + c\otimes \hat x\wedge\hat D\wedge\hat E
      + d\otimes \hat y\wedge\hat D\wedge\hat E)
        \\
        = b\otimes\hat x\wedge\hat y + c\otimes\hat x\wedge\hat D
          + d\otimes\hat y\wedge\hat D,
  \end{multlined} \\
  s^4(a\otimes\hat x\wedge\hat y\wedge\hat D\wedge\hat E)
        = -a\otimes\hat x\wedge\hat y\wedge\hat D
  \end{gather*}
and such that $d\circ s+s\circ d = \gamma$: this tells us that $\gamma$
induces the zero map on cohomology. Since our
ground field~$\kk$ has characteristic zero, this implies that the inclusion
$(A\otimes\Lambda V^*)_0\to A\otimes\Lambda V^*$ of the component of degree
zero of our complex $A\otimes\Lambda V^*$ is a quasi-isomorphism.

\paragraph\label{p:r3} \textbf{From now on and until the end of this
section, we will assume that $r\geq3$.} Let us write the complex
$(A\otimes\Lambda V^*)_0$ simply $\X$ and let us put $T=\kk[E]$, which
coincides with $A_0$. The complex~$\X$ has components
  \begin{gather*}
  \X^0 = A_0, \\
  \X^1 = A_1\otimes(\kk\hat x\oplus\kk\hat y)
        \oplus A_r\otimes\kk\hat D
        \oplus A_0\otimes\kk\hat E, \\
  \!\begin{multlined}[0.9\displaywidth]
  \X^2 = A_2\otimes\kk\hat x\wedge\hat y
        \oplus A_1\otimes(\kk\hat x\wedge\hat E\oplus\kk\hat y\wedge\hat E)
        \oplus A_r\otimes\kk\hat D\wedge\hat E
        \\
        \oplus A_{r+1}\otimes(\kk\hat x\wedge\hat D\oplus\kk\hat y\wedge\hat D),
  \end{multlined}\\
  \begin{multlined}[0.9\displaywidth]
  \X^3 = A_2\otimes\hat x\wedge\hat y\wedge\hat E
        \oplus A_{r+1}\otimes(\kk\hat x\wedge\hat D\wedge\hat E
        \oplus\kk\hat y\wedge\hat D\wedge\hat E) \\
        \oplus A_{r+2}\otimes\kk\hat x\wedge\hat y\wedge\hat D, 
  \end{multlined}
        \\
  \X^4 = A_{r+2}\otimes\hat x\wedge\hat y\wedge\hat D\wedge\hat E
  \end{gather*}
and, since $r>2$, we have
  \begin{align}
  &  A_0 = T,
  && A_1 = S_1T,
  && A_2 = S_2T, \\
  &  A_r = (S_r\oplus\kk D)T,
  && A_{r+1} = (S_{r+1}\oplus S_1D)T,
  && A_{r+2} = (S_{r+2}\oplus S_2D)T.
  \end{align}
In fact, this is where our assumption that $r\geq3$ intervenes: if $r\leq
2$, then these subspaces of~$A$ have different descriptions.

The differentials in~$\X$ can be computed to be given by 
  \begin{gather*}
  \delta^0(a) = x\tau_1(a)\otimes\hat x+y\tau_1(a)\otimes\hat y 
                        + D\tau_r(a)\otimes\hat D, \\[5pt]
  \delta^1(\phi a\otimes\hat x)
        = -\phi y\tau_1(a)\otimes\hat x\wedge\hat y
          -(F\phi_ya+\phi D\tau_r(a))\otimes\hat x\wedge\hat D
          +\nabla_x^{\phi a}(F)\otimes\hat y\wedge\hat D, \\
  \delta^1(\phi a\otimes\hat y)
        = \phi x\tau_1(a)\otimes\hat x\wedge\hat y
          +(\nabla_y^{\phi a}(F)-F\phi_ya-\phi D\tau_r(a))\otimes\hat y\wedge\hat D, \\
  \!\begin{multlined}[.9\displaywidth]
  \delta^1((\phi +\lambda D)a\otimes\hat D)
        = (\phi x\tau_1(a)+\lambda xD\tau_1(a))\otimes\hat x\wedge\hat D \\
          + (\phi y\tau_1(a)+\lambda F(\tau_1(a)-a)+\lambda yD\tau_1(a))
          \otimes\hat y\wedge\hat D,
  \end{multlined} \\
  \delta^1(a\otimes\hat E)
        = x\tau_1(a)\otimes\hat x\wedge\hat E
          +  y\tau_1(a)\otimes\hat y\wedge\hat E
          +  D\tau_r(a)\otimes\hat D\wedge\hat E, \\[5pt]
  \delta^2(\phi a\otimes\hat x\wedge\hat y)
        = (F\phi_ya+\phi D\tau_r(a)-\nabla_y^{\phi a}(F))
                \otimes\hat x\wedge\hat y\wedge\hat D, \\
  \!\begin{multlined}[0.9\displaywidth]
  \delta^2(\phi a\otimes\hat x\wedge\hat E)
        = -\phi y\tau_1(a)\otimes\hat x\wedge\hat y\wedge\hat E
          -(F\phi_ya+\phi D\tau_r(a))\otimes\hat x\wedge\hat D\wedge\hat E
          \\
          +\nabla_x^{\phi a}(F)\otimes\hat y\wedge\hat D\wedge\hat E,
  \end{multlined} \\
  \!\begin{multlined}[0.9\displaywidth]
  \delta^2(\phi a\otimes\hat y\wedge\hat E)
        = \phi x\tau_1(a)\otimes\hat x\wedge\hat y\wedge\hat E \\
          +(\nabla_y^{\phi a}(F)-F\phi_ya-\phi D\tau_r(a))
                \otimes\hat y\wedge\hat D\wedge\hat E, 
  \end{multlined} \\
  \delta^2((\phi+\psi D)a\otimes\hat x\wedge\hat D)
        = (-\phi y\tau_1(a)-\psi F(\tau_1(a)-a)-\psi yD\tau_1(a))
          \otimes\hat x\wedge\hat y\wedge\hat D, \\
  \delta^2((\phi+\psi D)a\otimes\hat y\wedge\hat D)
        = (\phi x\tau_1(a)+\psi xD\tau_1(a))
                \otimes\hat x\wedge\hat y\wedge\hat D, \\
  \!\begin{multlined}[.9\displaywidth]
  \delta^2((\phi+\lambda D)a\otimes\hat D\wedge\hat E)
        = (\phi x\tau_1(a)+\lambda xD\tau_1(a))
                \otimes\hat x\wedge\hat D\wedge\hat E \\
          +(\phi y\tau_1(a)+\lambda yD\tau_1(a)+\lambda F(\tau_1(a)-a))
          \otimes\hat y\wedge\hat D\wedge\hat E,
  \end{multlined} \\[5pt]
  \delta^3((\phi+\psi D)a\otimes\hat x\wedge\hat y\wedge\hat D)
        = 0, \\
  \delta^3(\phi a\otimes\hat x\wedge\hat y\wedge\hat E)
        = (F\phi_ya+\phi D\tau_r(a)-\nabla_y^{\phi a}(F))
          \otimes\hat x\wedge\hat y\wedge\hat D\wedge\hat E, \\
  \!\begin{multlined}[.9\displaywidth]
  \delta^3((\phi+\psi D)a\otimes\hat x\wedge\hat D\wedge\hat E)
        \\
        = -(\phi y\tau_1(a)+\psi y D\tau_1(a) +\psi F(\tau_1(a)-a))
                \otimes\hat x\wedge\hat y\wedge\hat D\wedge\hat E,
  \end{multlined} \\
  \delta^3((\phi+\psi D)a\otimes\hat y\wedge\hat D\wedge\hat E)
        = (\phi x\tau_1(a)+\psi xD\tau_1(a))
                \otimes\hat x\wedge\hat y\wedge\hat D\wedge\hat E.
  \end{gather*}
Here and below $\tau_t:T\to  T$ is the $\kk$-linear map such that
$\tau_t(E^n)=E^n-(E+t)^n$ for all $n\in\NN_0$, and $\phi$ and $\phi$ denote
homogeneous elements of~$r$ of appropriate degrees and $\lambda$ a scalar.

\paragraph We proceed to compute the cohomology of the complex~$\X$,
starting with degrees zero and four, for which the computation is almost
immediate. Indeed, since the kernel of $\tau_1$ and of $\tau_r$ is
$\kk\subseteq T$, it is clear that $H^0(\X)=\ker\delta^0=\kk$. On the other
hand, if $\psi\in S_2$ and $a\in T$, we can write $\psi=\psi_1x+\psi_2y$
for some $\psi_1$,~$\psi_2\in S_1$ and there is a $b\in T$ such that
$\tau_1(b)=a$, so that
  \[
  \delta^3(-\psi_2Db\otimes\hat x\wedge\hat D\wedge\hat E
     + \psi_1Db\otimes\hat y\wedge\hat D\wedge\hat E)
       = (\psi Da + \something{S_{r+2}T})
                \otimes\hat x\wedge\hat y\wedge\hat D\wedge\hat E.
  \]
Similarly, we have
  \(
  \delta^3(S_{r+1}T\otimes\hat x\wedge\hat D\wedge\hat E
           +S_{r+1}T\otimes\hat y\wedge\hat D\wedge\hat E)
        = S_{r+2}T\otimes\hat x\wedge\hat y\wedge\hat D\wedge\hat E
  \).
These two facts imply that the map~$\delta^3$ is surjective, so that
$H^4(\X)=0$.

\paragraph Let $\omega\in \X^1$ be a $1$-cocycle in~$\X$. There are then
$a$,~$b$,~$c$,~$d$,~$e$,~$f\in T$, $k\in\NN_0$ and
$\phi_0$,~\dots,~$\phi_k\in S_{r}$ such that either $k=0$ or $\phi_k\neq0$,
and
  \[
  \omega
    = (xa+yb)\otimes\hat x
     +(xc+yd)\otimes\hat y
     +\left(\sum_{i=0}^k\phi_i E^i+De\right)\otimes\hat D
     +f\otimes\hat E.
  \]
If $\bar e\in T$ is such that $\tau_r(\bar e)=e$, then by replacing
$\omega$ by $\omega-\delta^0(\bar e)$, which does not change the cohomology
class of~$\omega$, we can assume that $e=0$. The
formula for~$\delta^0$ then shows that $\omega$ is a coboundary iff it is
equal to zero. The coefficient of~$\hat x\wedge\hat y$ in~$\delta^1(\omega)$ is
  \[
  x^2\tau_1(c) + xy(\tau_1(d)-\tau_1(a)) - y^2\tau_1(b) = 0.
  \]
We therefore have $b$,~$c$,~$d-a\in\kk$. The coefficient of~$\hat D\wedge\hat E$,
on the other hand, is $D\tau_r(f)=0$, so that also $f\in\kk$; exactly the
same information comes from the vanishing of the coefficients of~$\hat x\wedge\hat
E$ and of~$\hat y\wedge\hat E$. Since $b\in\kk$, 
the coefficient of~$\hat x\wedge\hat D$ is
  \[
  -Fb-xD\tau_r(a)+\sum_{i=0}^k \phi_ix\tau_1(E^i) = 0.
  \]
We see that $\tau_r(a)=0$,
so that $a\in\kk$, and that $\sum_{i=0}^k \phi_ix\tau_1(E^i)=Fb$. This
implies that $k\leq1$, that $-\phi_1x=Fb$
and therefore, since $x$ is not a factor of~$F$ by hypothesis, that
$\phi_1=0$ and $b=0$.

Finally, using all the information we have so far, we can see that
the vanishing of the
coefficient of $\hat y\wedge\hat D$ in $\delta^1(\omega)$ implies that
$F_xxa + F_y(xc+yd) = Fd$. Together with Euler's relation
$F_xx+F_yy=(r+1)F$ this tells us that
  \[ \label{eq:qq}
  (cx+(d-a)y)F_y = (d-(r+1)a)F.
  \]
As $F$ is square-free, it follows\footnote{Suppose that $u=cx+(d-a)y$ is not zero.
Differentiating in~\eqref{eq:qq} with respect to~$y$, we find that
$-raF_y=uF_{yy}$. Since $x$ does not divide~$F$, we have $F_{yy}\neq0$, and
then $a\neq0$ and $u$ divides~$F_y$: from~\eqref{eq:qq} it follows then
that $u^2$ divides~$F$, since the left hand side of that equality is
non-zero, and this is absurd because $F$ is square-free.}
from this equality the polynomial $cx+(d-a)y$ is zero
so that $c=0$ and $d=a$ and, finally, that $a=0$.
We conclude in this way that the set of $1$-cocycles
  \[
  \phi\otimes\hat D+f\otimes\hat E, \qquad \phi\in S_r,f\in\kk
  \]
is a complete, irredundant set of representatives for the elements of
$H^1(\X)$.

\paragraph\label{p:3-cocycles} Let $\omega\in \X^3$ be a $3$-cocycle, so that
  \[
  \omega
        = a\otimes\hat x\wedge\hat y\wedge\hat D
        + b\otimes\hat x\wedge\hat y\wedge\hat E
        + c\otimes\hat x\wedge\hat D\wedge\hat E
        + d\otimes\hat y\wedge\hat D\wedge\hat E
  \]
for some $a\in (S_{r+2}\oplus S_2D)T$, $b\in S_2T$, $c$,~$d\in (S_{r+1}\oplus
S_1D)T$ and $\delta^3(\omega)=0$.
For all $\phi\in S_1$ and $e\in T$ we have
  \begin{gather}
  \!\begin{multlined}[.9\displaywidth]
  \delta^2(\phi e\otimes\hat x\wedge\hat E)
        = - \phi y\tau_1(e)\otimes\hat x\wedge\hat y\wedge\hat E
                + \something{A_{r+1}}\otimes\hat x\wedge\hat D\wedge\hat E \\
                + \something{A_{r+1}}\otimes\hat y\wedge\hat D\wedge\hat E 
  \end{multlined} \\
\shortintertext{and}
  \delta^2(\phi e\otimes\hat y\wedge\hat E)
        = \phi x\tau_1(e)\otimes\hat x\wedge\hat y\wedge\hat E
                + \something{A_{r+1}}\otimes\hat y\wedge\hat D\wedge\hat E,
  \end{gather}
so that by adding to $\omega$ an element of 
$\delta^2(S_1T\otimes\hat x\wedge\hat E+S_1T\otimes\hat y\wedge\hat E)$, 
which does not change the cohomology class of~$\omega$, we can
suppose that $b=0$. Similarly, for all $\phi\in S_2$ and all $e\in T$ we
have
  \begin{gather}
  \delta^2(\phi e\otimes\hat x\wedge\hat y) = 
        (\something{S_{r+2}T} + \phi D\tau_r(e))\otimes\hat x\wedge\hat y\wedge\hat D, \\
\intertext{and for all $\phi\in S_{r+1}$ and all $e\in T$ we have}
  \delta^2(\phi e\otimes\hat x\wedge\hat D) 
        = -\phi y \tau_1(e)\otimes\hat x\wedge\hat y\wedge\hat D
\shortintertext{and}
  \delta^2(\phi e\otimes\hat y\wedge\hat D) 
        = \phi x\tau_1(e)\otimes\hat x\wedge\hat y\wedge\hat D.
  \end{gather}
Using this we see that, up to changing~$\omega$ by adding to it a
$3$-coboundary, we can suppose that $a=0$. Finally, for each $\phi\in S_r$
and all $e\in T$ we have
  \begin{gather}
  \delta^2(\phi e\otimes\hat D\wedge\hat E)
        = \phi x\tau_1(e)\otimes\hat x\wedge\hat D\wedge\hat E
          + \something{A_{r+1}} \otimes\hat y\wedge\hat D\wedge\hat E,
          \\
  \delta^2(De\otimes\hat D\wedge\hat E)
        = xD\tau_1(e)\otimes\hat x\wedge\hat D\wedge\hat E
          + \something{A_{r+1}}\otimes\hat y\wedge\hat D\wedge\hat E
\shortintertext{and}
  \delta^2(-y\otimes\hat x\wedge\hat E+\bar FE\otimes\hat D\wedge\hat E)
        = y^{r+1}\otimes\hat x\wedge\hat D\wedge\hat E
          +\something{A_{r+1}}\otimes\hat y\wedge\hat D\wedge\hat E,
  \end{gather}
so we can also suppose that $c\in y^{r+1}ET+yDT$.

There are $l\geq0$, $\lambda_1$,~\dots,~$\lambda_l$, $\mu_0$,~\dots,~$\mu_l\in\kk$,
$\phi_0$,~\dots,~$\phi_l\in S_{r+1}$, $\psi_0$,~\dots,~$\psi_l\in S_1$,
$\zeta_0$,~\dots,~$\zeta_l\in S_1$ such that
$c=\sum_{i=1}^l\lambda_i y^{r+1}E^i+\sum_{i=0}^l\mu_i yDE^i$ and
$d=\sum_{i=0}^l(\phi_i+\psi_iD)E^i$. The vanishing of
$\delta^3(\omega)$ means precisely that
  \begin{equation*}
  \sum_{i=1}^l\lambda_iy^{r+2}\tau_1(E^i)
  +
  \sum_{i=0}^l
        \Bigl(
        \mu_iy^2D\tau_1(E^i)
        -\mu_iyF(E+1)^i
        -\phi_ix\tau_1(E^i)
        -\psi_ixD\tau_1(E^i)
        \Bigr)
        = 0.
  \end{equation*}
The left hand side of this equation is an element of $S_{r+2}T\oplus
S_2DT$. The component in~$S_2DT$ is
  \(
  \sum_{i=0}^l(\mu_iy^2-\psi_ix)D\tau_1(E^i) = 0
  \)
and therefore $\mu_i=\psi_i=0$ for all $i\in\{1,\dots,l\}$. On the other
hand, the component in~$S_{r+2}T$ is
  \[
  \sum_{i=1}^l\lambda_iy^{r+2}\tau_1(E^i)
  - \mu_0yF - \sum_{i=0}^l\phi_ix\tau_1(E^i) = 0.
  \]
This implies that $\lambda_iy^{r+2}-\phi_ix=0$ if $i\in\{2,\dots,l\}$,
so that $\lambda_i=\phi_i=0$ for such $i$, and then the equation reduces to
$\lambda_1y^{r+2}+\mu_0yF-\phi_1x=0$.
Recalling from~\pref{p:start} that $F=y^{r+1}+x\bar F$,
we deduce from this that $\lambda_1=-\mu_0$ and $\phi_1=\mu_0y\bar F$.
We conclude in this way that every $3$-cocycle is cohomologous to one of
the form
  \[ \label{eq:q3}
  (\mu_0yD-\mu_0y^{r+1}E)\hat x\wedge\hat D\wedge\hat E
  + (\phi_0+\psi_0D+\mu_0y\bar FE)\hat y\wedge\hat D\wedge\hat E
  \]
with $\mu_0\in\kk$, $\phi_0\in S_{r+1}$ and $\psi_0\in S_1$, and a direct
computation shows that moreover every $3$-cochain of this form is a $3$-cocycle.

Let now $\eta$ be a $2$-cochain~$\eta$ in~$\X$, so that
  \begin{multline}
  \eta = \something{A_2\otimes\kk\hat x\wedge\hat y
        \oplus A_{r+1}\otimes(\kk\hat x\wedge\hat D\oplus\kk\hat y\wedge\hat D)} 
        + u\otimes\hat x\wedge\hat E 
        + v\otimes\hat y\wedge\hat E \\
        + w\otimes\hat D\wedge\hat E
  \end{multline}
with $u$,~$v\in A_1$ and $w\in A_{r}$, and let us
suppose that $\delta^2(\eta)$ is equal to the $3$-cocycle~\eqref{eq:q3}.
There are $l\geq0$, $\alpha_0$,~\dots,~$\alpha_l$, $\beta_0$,~\dots,~$\beta_l\in S_1$,
$\gamma_0$,~\dots,~$\gamma_l\in S_r$ and $\xi_0$,~\dots,~$\xi_l\in\kk$
such that $u=\sum_{i=0}^l\alpha_iE^i$, $v=\sum_{i=0}^l\beta_iE^i$ and
$w=\sum_{i=0}^l(\gamma_i+\xi_iD)E^i$. The coefficient of~$\hat x\wedge\hat
y\wedge\hat E$ in~$\delta^2(\eta)$ must be equal to zero, so that
  \[
  \sum_{i=0}^l(-\alpha_iy+\beta_ix)\tau_1(E^i) = 0,
  \]
and this implies that there are scalars $\rho_1$,~\dots,~$\rho_l\in\kk$ such that
$\alpha_i=\rho_i x$ and $\beta_i=\rho_i y$ for all $i\in\{1,\dots,l\}$.

Looking now at the coefficient of $\hat x\wedge\hat D\wedge\hat E$ in~$\delta^2(\eta)$
and comparing with~\eqref{eq:q3} we find that
  \[ \label{eq:qt}
  \sum_{i=0}^l\left(-F\alpha_{iy}E^i-\alpha_iD\tau_r(E^i)
        + \gamma_ix\tau_1(E^i) + \xi_i xD\tau_1(E^i)\right)
        = \mu_0yD-\mu_0y^{r+1} E.
  \]
This is an equality of two elements of $ST\oplus SDT$. Considering the
components in~$DT$, we find that
  \(
  xD\sum_{i=1}^l(-\rho_i\tau_r(E^i)+\xi_i\tau_1(E^i)) = \mu_0yD
  \), 
and this tells us that $\mu_0=0$ and that
  \[ \label{eq:q4}
  \sum_{i=1}^l\left(-\rho_i\tau_r(E^i)+\xi_i\tau_1(E^i)\right) = 0.
  \]
On the other hand, as the components in~$ST$ of the two sides
of~\eqref{eq:qt} are equal, we have
  \[
  -F\alpha_{0y} +
  \sum_{i=0}^l\gamma_ix\tau_1(E^i)
        = 0,
  \]
so that $\gamma_i=0$ for all $i\in\{2,\dots,l\}$ and
$F\alpha_{0y}+\gamma_1x=0$. As $x$ does not divide~$F$, we must have
$\alpha_{0y}=0$ and $\gamma_1=0$; in particular, there is a $\rho_0\in\kk$
such that $\alpha_0=\rho_0 x$.

Finally, considering the coefficient of~$\hat y\wedge\hat D\wedge\hat E$
of~$\delta^2(\eta)$ and of~\eqref{eq:q3} we see that
  \begin{multline*}
  \sum_{i=0}^l\Bigl(\nabla_x^{\alpha_iE^i}(F)
        + \nabla_y^{\beta_iE^i}(F)-F\beta_{iy}E^i - \beta_iD\tau_r(E^i) \\
        + \gamma_iy\tau_1(E^i) + \xi_iyD\tau_1(E^i) - \xi_iF(E+1)^i
        \Bigr)
        = \phi_0 + \psi_0D,
  \end{multline*}
which at this point we can rewrite (using in the process the
equality~\eqref{eq:q4} above and the fact that
$\nabla_x^{xE^i}(F)+\nabla_x^{yE^i}(F)=F\sum_{t=0}^r(E+t)^i$) as
  \[
  \rho_0xF_x+\beta_0F_y
        -F\left(\beta_{0y}+\xi_0-
        \sum_{i=1}^l\left(\rho_i\sum_{t=1}^r(E+t)^i-\xi_i(E+1)^i\right)
        \right)
        = \phi_0 + \psi_0D.
  \]
It follows at once that $\psi_0=0$ and that, in fact,
  \[
  \rho_0xF_x+\beta_0F_y
        -F\left(\beta_{0y}+\xi_0-
        \sum_{i=1}^l\left(\rho_i\sum_{t=1}^rt^i-\xi_i\right)
        \right)
        = \phi_0.
  \]
The polynomial $\phi_0$ is then in the linear span of $xF_x$,
$xF_y$, $yF_y$ and $F$ inside~$S_{r+1}$. Euler's relation implies that
already the first three polynomials span this subspace, and we have
  \[\label{eq:q4p}
  \begin{lgathered}
   \delta(x\otimes\hat x\wedge\hat E) = xF_x\otimes\hat y\wedge\hat D\wedge\hat E, \\
   \delta(x\otimes\hat y\wedge\hat E) = xF_y\otimes\hat y\wedge\hat D\wedge\hat E, \\
   \delta(y\otimes\hat y\wedge\hat E-D\otimes\hat D\hat E) 
                = yF_y\otimes\hat y\wedge\hat D\wedge\hat E.
  \end{lgathered}
  \]
We conclude in this way that the only $3$-coboundaries among the cocycles
of the form~\eqref{eq:q3} are the linear combinations of the right hand
sides of the equalities~\eqref{eq:q4p}; these three cocycles are, moreover,
linearly independent. This means that there is an isomorphism
  \[
  H^3(\X) \cong
        \kk\omega_3
        \oplus S_1D\otimes\hat y\wedge\hat D\wedge\hat E
        \oplus \frac{S_{r+1}}{\lin{xF_x,xF_y,yF_y}}\otimes\hat y\wedge\hat D\wedge\hat E,
  \]
with
  \[
  \omega_3=
    (yD-y^{r+1}E)\otimes\hat x\wedge\hat D\wedge\hat E
    +y\bar FE\otimes\hat y\wedge\hat D\wedge\hat E,
  \]
and that, in particular, $\dim H^3(\X)=r+2$, since the denominator
appearing in the right hand side of this isomorphism is a $3$-dimensional
vector space ---this follows at once from Lemma~\pref{lemma:indep}.

\paragraph We consider now a $2$-cocycle $\omega\in\X^2$ and $a\in S_2T$,
$b$,~$c\in S_1T$, $d$,~$e\in S_{r+1}T\oplus S_1DT$ and $f\in S_rT\oplus DT$
such that
  \[
  \omega = a\otimes\hat x\wedge\hat y+b\otimes\hat x\wedge\hat E
           + c\otimes\hat y\wedge\hat E
           + d\otimes\hat x\wedge\hat D
           + e\otimes\hat y\wedge\hat D
           + f\otimes\hat E\wedge\hat D.
  \]
Adding to~$\omega$ an element of $\delta^1(T\otimes\hat E)$,
we can assume that $f\in S_rT$; adding an element of $\delta^1(S_1T\otimes\hat
x\oplus S_1T\otimes\hat y)$, we can suppose that $a=0$; finally, adding
an element of $\delta^1((S_rT\oplus DT)\otimes\hat D)$ we can suppose that $d\in
y^{r+1}T\oplus yDT$.
In this situation, there are an integer $l\geq0$,
$\alpha_0$,~\dots,~$\alpha_l$, $\beta_0$,~\dots,~$\beta_l\in S_1$,
$\lambda_0$,~\dots,~$\lambda_l$, $\mu_0$,~\dots,~$\mu_l\in\kk$,
$\phi_0$,~\dots,~$\phi_l\in S_{r+1}$, $\psi_0$,~\dots,~$\psi_l\in S_1$ and
$\xi_0$,~\dots,~$\xi_l\in S_r$ such that $b=\sum_{i=0}^l\alpha_iE^i$,
$c=\sum_{i=0}^l\beta_iE^i$, $d=\sum_{i=0}^l(\lambda_iy^{r+1}+\mu_iyD)E^i$,
$e=\sum_{i=0}^l(\phi_i+\psi_iD)E^i$ and $f=\sum_{i=0}^l\xi_iE^i$. As
  \[
  \delta^1(-y\otimes\hat x+\bar FE\otimes\hat D)
        = y^{r+1}\otimes\hat x\wedge\hat D
          + \something{S_{r+1}}\otimes\hat y\wedge\hat D,
  \]
we can assume that $\lambda_0=0$.

The coefficient of $\hat x\wedge\hat y\wedge\hat E$ in~$\delta^2(\omega)$ is
$\sum_{i=0}^l(-\alpha_iy+\beta_ix)\tau_1(E^i)=0$, and this implies that
there are scalars $\rho_1$,~\dots,~$\rho_l\in\kk$ such that
$\alpha_i=\rho_ix$ and $\beta_i=\rho_iy$ for each $i\in\{1,\dots,l\}$. The
coefficient of $\hat x\wedge\hat D\wedge\hat E$ in $\delta^2(\omega)$ is
  \[ \label{eq:w1}
  \sum_{i=0}^l
        \bigl(-F\alpha_{iy}E^i-\alpha_iD\tau_r(E^i)+\xi_ix\tau_1(E^i)\bigr)
        = 0.
  \]
It follows that $\sum_{i=0}^l\alpha_iD\tau_r(E^i)=0$, so that
$\alpha_1=\cdots=\alpha_l=0$; as a consequence of this, we have that
$\rho_1=\cdots=\rho_l=0$ and $\beta_1=\cdots=\beta_l=0$. The
equality~\eqref{eq:w1} also tells us that
$-F\alpha_{0y}+\sum_{i=0}^l\xi_ix\tau_1(E^i)=0$, and from this we see that
$\xi_2=\cdots=\xi_l=0$ and $-F\alpha_{0y}-\xi_1x=0$, so that
$\alpha_{0y}=0$ and $\xi_1=0$, since $x$ does not divide~$F$. In
particular, there is a $\rho_0\in\kk$ such that $\alpha_0=\rho_0x$.

The coefficient of $\hat y\wedge\hat D\wedge\hat E$
in~$\delta^2(\omega)$ is
  \begin{align}
  \MoveEqLeft
  \sum_{i=0}^l
        \Bigl(
        \nabla_x^{\alpha_iE^i}(F) + \nabla_y^{\beta_iE^i}(F)
        - F \beta_{iy}E^i - \beta_iD\tau_r(E^i)
        +\xi_iy\tau_1(E^i)
        \Bigr) \\
    &= \rho_0xF_x+\beta_0F_y-\beta_{0y}F \\
    &= \bigl(\rho_0-(r+1)^{-1}\beta_{0y}\bigr)xF_x
       + \bigl(\beta_{0x}x+(1-(1+r)^{-1})\beta_{0y}y)F_y
     = 0,
  \end{align}
and our Lemma~\pref{lemma:indep} implies then that $\beta_0=0$ and $\rho_0=0$.
Finally, we consider the coefficient of $\hat x\wedge\hat y\wedge\hat D$:
  \[
  \sum_{i=0}^l
        \Bigl(
        -\lambda_iy^{r+2}\tau_1(E^i)
        +\mu_iyF(E+1)^i
        -\mu_iy^2D\tau_1(E^i)
        +\phi_ix\tau_1(E^i)+\psi_ixD\tau_1(E^i)
        \Bigr)
        = 0.
  \]
Looking at the terms involving~$D$ in this equation, we see that
  \[
  \sum_{i=0}^l(-\mu_iy^2+\psi_ix)D\tau_1(E^i))=0,
  \]
so $\mu_1=\cdots=\mu_l=0$ and $\psi_1=\cdots=\psi_l=0$. The terms not
involving~$D$ add up to
  \[
  \mu_0yF+\sum_{i=0}^l(-\lambda_iy^{r+2}+\phi_ix)\tau_1(E^i)=0,
  \]
so that $\lambda_2=\cdots=\lambda_l=0$, $\phi_2=\cdots=\phi_l=0$ and
$\mu_0yF+\lambda_1y^{r+1}-\phi_1x=0$, which implies that $\lambda_1=-\mu_0$
and $\phi_1=\mu_0y\bar F$.

After all this, we see that every $2$-cocycle in our complex is
cohomologous to one of the form
  \[ \label{eq:w2}
  (\mu_0yD-\mu_0y^{r+1}E)\hat x\wedge\hat D
  + (\phi_0+\psi_0D+\mu_0y\bar FE)\hat y\wedge\hat D
  + \xi_0\hat D\wedge\hat E
  \]
with $\mu_0\in\kk$, $\phi_0\in S_{r+1}$, $\psi_0\in S_1$ and
$\xi_0\in S_r$. Computing we find that all elements of this form are in
fact $2$-cocycles.

Let us now suppose that the cocycle~\eqref{eq:w2}, which we call again~$\omega$,
is a coboundary, so that
there exist $k\geq0$, $\alpha_0$,~\dots,~$\alpha_k$,
$\beta_0$,~\dots,~$\beta_k\in S_1$, $\sigma_1$,~\dots,~$\sigma_k\in S_r$,
$\zeta_0$,~\dots,~$\zeta_k\in\kk$ and $u\in T$ such that if
  \[
  \eta = \sum_{i=0}^k\alpha_iE^i\hat x
        +\sum_{i=0}^k\beta_iE^i\hat y
        +\sum_{i=0}^k(\sigma_i+\zeta_iD)E^i\hat D
        +u\hat E,
  \]
we have $\delta^1(\eta)=\omega$. The coefficient of~$\hat D\wedge\hat E$
in~$\delta^1(\eta)$ is $D\tau_r(u)$ so, comparing with~\eqref{eq:w2}, we
see that we must have $\xi_0=0$ and $u\in\kk$; it follows from this that
the coefficients of $\hat E\wedge\hat E$ and of $\hat y\wedge\hat E$
in~$\delta^1(\eta)$ vanish. On the other hand, the coefficient of~$\hat
x\wedge\hat y$ in~$\delta^1(\eta)$ is
$\sum_{i=0}^k(-\alpha_iy+\beta_ix)\tau_1(E^i)$: as this has to be zero, we
see that there exist $\rho_1$,~\dots,~$\rho_k\in\kk$ such that
$\alpha_i=\rho_ix$ and $\beta_i=\rho_iy$ for each $i\in\{1,\dots,k\}$.

The coefficient of~$\hat x\wedge\hat D$ in~$\delta^1(\eta)$ is
  \[ \label{eq:w3}
  \sum_{i=0}^k
        \bigl(
        -F\alpha_{iy}E^i-\alpha_iD\tau_r(E^i)+\sigma_ix\tau_1(E^i)
        +\zeta_ixD\tau_1(E^i)
        \bigr) = \mu_0yD-\mu_0y^{r+1}E.
  \]
This means, first, that
  \(
  \sum_{i=1}^k
        \bigl(
        -\rho_ixD\tau_r(E^i)+\zeta_ixD\tau_1(E^i)
        \bigr)
        =\mu_0yD
  \)
and this is only possible if $\mu_0=0$ and
  \[ \label{eq:w0}
  \sum_{i=1}^k\bigl(-\rho_i\tau_r(E^i)+\zeta_i\tau_1(E^i)\bigr)=0.
  \]
Second, the equality~\eqref{eq:w3} implies that
  \[
  \sum_{i=0}^k
        \bigl(
        -F\alpha_{iy}E^i+\sigma_ix\tau_1(E^i)
        \bigr)
        = -F\alpha_{0y}+\sum_{i=1}^k\sigma_ix\tau_1(E^i)
        = 0,
  \]
so that $\sigma_2=\cdots=\sigma_k=0$ and
$F\alpha_{0y}+\sigma_1x=0$, which tells us that
$\sigma_1=0$ and $\alpha_{0y}=0$; there is then a $\rho_0\in\kk$ such that
$\alpha_0=\rho_0x$.

Finally, the coefficient of $\hat y\wedge\hat D$ in~$\delta^1(\eta)$ is
  \begin{multline}
  \sum_{i=0}^k
        \bigl(
        \nabla_x^{\alpha_iE^i}(F)+\nabla_y^{\beta_iE^i}(F)
        - F\beta_{iy}E^i
        - \beta_iD\tau_r(E^i)
        +\sigma_iy\tau_1(E^i)
        \\
        -\zeta_iF(E+1)^i+\zeta_iyD\tau_1(E^i)
        \bigr)
        = \phi_0+\psi_0D.
  \end{multline}
Looking only at the terms which are in $S_1DT$, we see that
  \[
  yD\sum_{i=1}^k(-\rho_i\tau_r(E^i)+\zeta_i\tau(E^i))=\psi_0D
  \]
and, in view of~\eqref{eq:w0}, it follows from this that $\psi_0=0$. The
terms in~$S_{r+1}T$, on the other hand, are
  \[
  \rho_0xF_x+\beta_0F_y
    +F
    \left(
    -\beta_{0y}-\zeta_0
    +\sum_{i=0}^k
        \Bigl(
        \rho_i\sum_{t=1}^r(E+t)^i-\zeta_i(E+1)^i
        \Bigr)
    \right) = \phi_0,
  \]
and proceeding as before we see that $\phi_0$ is in the linear span of
$xF_x$, $xF_y$ and $yF_y$. Computing, we find that
  \begin{gather}
  \delta^1(x\otimes\hat x)=xF_x\otimes\hat y\wedge\hat D, \\
  \delta^1(x\otimes\hat y)=xF_y\otimes\hat y\wedge\hat D, \\
  \delta^1(y\otimes\hat y-D\hat D)=yF_y\otimes\hat y\wedge\hat D.
  \end{gather}
We thus conclude that there is an isomorphism
  \[
  H^2(\X)
        \cong
        \kk\omega_2
        \oplus \frac{S_{r+1}}{\lin{xF_x,xF_y,yF_y}}\otimes\hat y\wedge\hat D
        \oplus S_1D\otimes\hat y\wedge\hat D
        \oplus S_r\otimes\hat D\wedge\hat E,
  \]
with $\omega_2=(yD-y^{r+1}E)\otimes\hat x\wedge\hat D+y\bar FE\otimes\hat
y\wedge\hat D$, and that, in particular, the dimension of~$H^2(\X)$ is~$2r+3$.

\paragraph We can summarize our findings as follows:

\begin{Proposition*}\plabel{prop:hh}\pushQED{\qed}
Suppose that $r\geq3$.
For all $p\geq4$ we have $\HH^p(A)=0$. There are
isomorphisms 
  \begin{gather}
  \HH^0(A) \cong \kk, \\
  \HH^1(A) \cong S_r\otimes\hat D\oplus\kk\otimes\hat E, \\
  \HH^2(A) \cong
        \kk\omega_2
        \oplus \frac{S_{r+1}}{\lin{xF_x,xF_y,yF_y}}\otimes\hat y\wedge\hat D
        \oplus S_1D\otimes\hat y\wedge\hat D
        \oplus S_r\otimes\hat D\wedge\hat E, \\
  \HH^3(A) \cong
        \kk\omega_3
        \oplus \frac{S_{r+1}}{\lin{xF_x,xF_y,yF_y}}\otimes\hat y\wedge\hat D\wedge\hat E
        \oplus S_1D\otimes\hat y\wedge\hat D\wedge\hat E, 
  \end{gather}
with 
  \begin{gather}
  \omega_2=(yD-y^{r+1}E)\otimes\hat x\wedge\hat D+y\bar FE\otimes\hat y\wedge\hat D, \\
  \omega_3=
    (yD-y^{r+1}E)\otimes\hat x\wedge\hat D\wedge\hat E
    + y\bar FE\otimes\hat y\wedge\hat D\wedge\hat E.
  \end{gather}
The Hilbert series of the Hochschild cohomology of~$A$ is
  \begin{align}
  h_{\HH^\bullet(A)}(t) 
       &= 1 + (r+2)t + (2r+3)t^2 + (r+2)t^3 \\
       &= (1+t)(1+(r+1)t+(r+2)t^2). \qedhere
  \end{align}
\end{Proposition*}

In fact, in each of the isomorphisms appearing in the statement of the
proposition we have given a set of representing cocycles. This will be
important in what follows, when we compute the Gerstenhaber algebra
structure on the cohomology of~$A$.

We have chosen a system of coordinates in the vector space
containing the arrangement~$A$ in such a way that one of the lines is given
by the equation $x=0$. This was useful in picking a basis for the $S$-module
of derivations~$\Der(\A)$ and, as a consequence, obtaining a presentation
of the algebra~$A$ amenable to the computations we wanted to carry out, but
the unnaturality of our choice is reflected in the rather unpleasant form
of the representatives that we have found for cohomology classes ---a
consequence of the combination of the truth of Hermann Weyl's dictum that
the introduction of coordinates is an act of violence together with that of
the everyday observation that violence does not lead to anything good. In
the next section we will be able to obtain a more natural description.

\paragraph In Proposition~\pref{prop:hh} we considered only line
arrangements with $r\geq3$, that is, with at least $5$ lines. As we
explained in~\pref{p:r3}, without the restriction the method of
calculation that we followed has to be modified, and it turns out that this
is not only a technical difference: the actual results are different. Let
us describe what happens, starting with the factorizable cases:
\begin{itemize}

\item If there are no lines, so that $r=-2$, the arrangement is empty and
$\D(\A)$ is the second Weyl algebra $\kk[x,y,\partial_x,\partial_y]$.

\item If there is one line, then $\D(\A)$ is
$\kk[x,y,x\partial_x,\partial_y]$ and this is isomorphic to $\U(\mathfrak
s)\otimes A_1$, with $\U(\mathfrak s)$ the enveloping algebra of the
non-abelian $2$-dimensional Lie algebra~$\mathfrak s$ and $A_1$ the first
Weyl algebra.

\item If there are two lines, so that $r=0$, then $\D(\A)$ is
$\kk[x,y,x\partial_x,y\partial_y]$, which is isomorphic to $\U(\mathfrak
s)\otimes\U(\mathfrak s)$. 

\end{itemize}
The Hochschild cohomology of the Weyl algebras is well-known ---for
example, from~\cite{Sridharan}--- as is that of~$\U(s)$. Using this and Künneth's
formula we find that when $-2\leq r\leq0$ we have for all $i\in\NN_0$.
  \[
  \dim\HH^i(\D(\A)) = \binom{r+2}{i}.
  \]
Finally, we have the cases of three and four lines. Up to isomorphism of
arrangements, one can assume that the defining polynomials are $Q=xy(x-y)$
and $Q=xy(x-y)(x-\lambda y)$ for some $\lambda\in\kk\setminus\{0,1\}$,
respectively. One can compute the cohomology of $\D(\A)$ in these cases
along the lines of what we did above, but the computation is surprisingly
much more involved. We have done the computation using an alternative, much
more efficient approach ---using a spectral sequence that computes in
general the Hochschild cohomology of the enveloping algebra of a
Lie--Rinehart pair--- on which we will report in an upcoming paper. Let us
for now simply summarize the result: when $r$ is $2$ or~$3$, the Hilbert
series of $\HH^\bullet(A)$ is
  \[
  h_{\HH^\bullet(A)}(t) 
        = 1 + (r+2)t + (2r+4)t^2 + (r+3)t^3.
  \]
This differs from the general case of Proposition~\pref{prop:hh} in the
coefficients of~$t^2$ and~$t^3$.

For our immediate purposes, we remark that in all cases $\HH^1(\D(\A))$ has
dimension equal to the number of lines in the arrangement~$\A$, and that
its concrete description is the same in all cases.

\section{The Gerstenhaber algebra structure on 
\texorpdfstring{$\HH^\bullet(\D(\A))$}{HH(D(A))}}
\label{sect:cup}

\paragraph\label{p:comparison} 
Let $\B A$ be the usual bar resolution for $A$ as an $A$-bimodule.
There is a morphism of complexes $\phi:\P\to\B A$ over
the identity map of~$A$ such that $\phi=\phi_K+\phi_N$ with
$\phi_K$,~$\phi_N:\P\to\B A$ maps of $A$-bimodules such that
  \[
  \phi_K(1|v_1\wedge\cdots\wedge v_p|1)
        = \sum_{\pi\in S_p}(-1)^{\epsilon(\pi)}
                1|v_{\pi(1)}|\cdots|v_{\pi(p)}|1,
  \]
whenever $p\geq0$ and $v_1$,~\dots,~$v_p\in V$,
with the sum running over permutations of degree~$p$,
and
  \begin{gather}
  \phi_N(1|1)
        = 0; \\
  \phi_N(1|v|1)
        = 0, \quad\forall v\in V; \\
  \!\begin{multlined}[.9\displaywidth]
  \phi_N(1|x\wedge y|1)
        = \phi_N(1|x\wedge E|1)
        = \phi_N(1|y\wedge E|1)
        = \phi_N(1|x\wedge D|1) \\
        = \phi_N(1|D\wedge E|1)
        = 0;
  \end{multlined}\\
  \phi_N(1|y\wedge D|1)
        = q\pt1|\bar q\pt2|q\pt3|1 - F|1|1|1; \\
  \phi_N(1|x\wedge y\wedge E|1)
        = \phi_N(1|x\wedge D\wedge E|1)
        = 0; \\
  \!\begin{multlined}[.9\displaywidth]
  \phi_N(1|x\wedge y\wedge D|1)
        = q\pt1|\bar q\pt2|q\pt3|x|1
                - q\pt1|\bar q\pt2|x|q\pt3|1
                + q\pt1|x|\bar q\pt2|q\pt3|1 \\
                - F|x|1|1|1 - F|1|1|x|1;
  \end{multlined} \\
  \!\begin{multlined}[.9\displaywidth]
  \phi_N(1|y\wedge D\wedge E|1)
        = q\pt1|\bar q\pt2|q\pt3|E|1
                - q\pt1|\bar q\pt2|E|q\pt3|1
                + q\pt1|E|\bar q\pt2|q\pt3|1 \\
                - F|E|1|1|1 - F|1|1|E|1.
  \end{multlined}
  \end{gather}
Here $q\pt{1}|\bar q\pt{2}|q\pt{3}$ denotes the element $\nabla(F)\in
S\otimes S_1\otimes S$, with an omitted sum.

On the other hand, there is a morphism of complexes of $A$-bimodules
$\psi:\B A\to\P$ over the identity map of~$A$ such that
  \begin{gather}
  \psi_0(1|1) = 1|1, \\
  \psi_1(1|w|1) = w\pt1|w\pt2|w\pt 3,
        \qquad\text{for all standard monomials $w$}; \\
  \psi_2(1|yD|y|1) = -y|y\wedge D|1 - q\pt{1}|q\pt{2}\wedge y|q\pt{3}; \\
  \psi_2(1|y^{r+1}E|y|1) = -y^{r+1}|y\wedge E|1; \\
  \psi_2(1|E|w|1) = -w\pt1|w\pt2\wedge E|w\pt3
        \qquad\text{for all standard monomials $w$}; \\
  \psi_2 (1|v|w|1) = - 1|w\wedge v |1, \label{eq:psi2vw}
	\qquad\text{if $v,w\in\{x,y,D,E\}$ and $vw$ is not standard}; \\
  \psi_2(1|w|x|1) = -w\pt1|x\wedge w\pt2|w\pt3
        \qquad\text{for all standard monomials $w$}; \\
\shortintertext{and}
  \psi_2(1|u|v|1) = 0 \label{eq:psi2uv}
  \end{gather}
whenever $u$ and $v$ are standard monomials of~$A$ such that the
concatenation $uv$ is also a standard monomial. This morphism~$\psi$ can be
taken ---and we will take it--- to be normalized, so that it vanishes on
elementary tensors of~$\B A$ with a scalar factor.

\paragraph We need the comparison morphisms that we have just described in
order to compute the Gerstenhaber bracket on~$\HH^\bullet(A)$, but we start
with a more immediate application: obtaining a natural basis of the first
cohomology space~$\HH^1(A)$.

\begin{Proposition*}\plabel{prop:partial}
\begin{thmlist}\fixhspace

\item If $\alpha$ is a non-zero element of~$S_1$ that divides $Q$, then
there exists a unique derivation $\partial_\alpha:A\to A$ such that
$\partial_\alpha(f)=0$ for all $f\in S$ and
  \[
  \partial_\alpha(\delta) = \frac{\delta(\alpha)}{\alpha}
  \]
for all $\delta\in\Der(\A)$. 

\item If $Q=\alpha_0\dots\alpha_{r+1}$ is a factorization of~$Q$ as a
product of elements of~$S_1$, then the cohomology classes of the $r+2$
derivations $\partial_{\alpha_0}$,~\dots,~$\partial_{\alpha_{r+1}}$ of~$A$
freely span the vector space~$\HH^1(A)$.

\end{thmlist}
\end{Proposition*}

Here we are viewing $\HH^1(A)$ as the vector space of outer
derivations of~$A$, as usual. It should be noticed that the
derivation~$\partial_\alpha$ associated to a linear factor of~$Q$ does not
change if we replace~$\alpha$ by one of its non-zero scalar multiples: this
means that the basis of~$\HH^1(A)$ is really indexed by the lines of the
arrangement~$\A$.

\begin{proof}
\thmitem{1} Let us fix a non-zero element~$\alpha$ in~$S_1$ dividing~$Q$.
There is at most one derivation~$\partial_\alpha:A\to A$ as in the
statement of the proposition simply because the algebra~$A$ is generated by
the set $S\cup\Der(\A)$. In order to prove that there is such a derivation,
we need only recall from \cite{OT}*{Proposition 4.8} that
$\delta(\alpha)\in\alpha S$ for all $\delta\in\Der(\A)$ and check that the
candidate derivation respects the relations~\eqref{eq:relations}
of~\pref{p:start} that present the algebra~$A$.

\thmitem{2} We need to pass from the description of~$\HH^1(A)$ as the
space of outer derivations to its description in terms of the complex~$\X$
that was used to compute it: we do this with the comparison morphism 
$\phi:\P\to\B A$ over the identity map that we described
in~\pref{p:comparison}.
If $\delta:A\to A$ is a
derivation of~$A$ and $\tilde\delta:A\otimes A\otimes A\to A$ is the map
such that $\tilde\delta(a\otimes b\otimes c)=a\delta(b)c$ for all
$a$,~$b$,~$c\in A$, which is a $1$-cocycle on~$\B A$, the composition
$\bar\delta\circ\phi_1:A\otimes V\otimes A\to A$ is a $1$-cocycle in the
complex~$\hom_{A^e}(\P,A)$ whose cohomology class corresponds to~$\delta$
in the usual description of~$\HH^1(A)$ as the space of outer derivations
of~$A$. In the notation that we used in~\pref{p:hhdifferentials}, this
cohomology class is that of 
  \[
  \delta(x)\otimes\hat x + \delta(y)\otimes\hat y
    + \delta(D)\otimes\hat D + \delta(E)\otimes\hat E
    \in A\otimes\hat V.
  \]
Using this, we can now prove the second part of the proposition.
We can suppose without loss of generality that $\alpha_0=x$,
and then the class of $\delta_{\alpha_0}$ in~$\HH^1(A)$ is that of
  \[
  1\otimes \hat E.
  \]
On the other hand, for each
$i\in\{1,\dots,r+1\}$, computing we find that the class
of~$\partial_{\alpha_i}$ is
  \[
  \alpha_{iy}\frac{F}{\alpha_i}\otimes\hat D + 1\otimes\hat E.
  \]
It follows easily from the second part of Lemma~\pref{lemma:indep} that
these $r+2$ classes span~$\HH^1(A)$ and, since the dimension of
this space is exactly~$r+2$, do so freely.
\end{proof}

\goodbreak

\subsection*{The cup product}

\paragraph\label{prop:cup} We describe the associative algebra
structure on~$\HH^\bullet(A)$ given by the cup product.

\begin{Proposition*}
The cup product on $\HH^\bullet(A)$ is such that
  \begin{alignat}{3}
  & S_r\otimes\hat D\smile S_r\otimes\hat D = 0; \\
  & \phi\hat D\smile\hat E = \phi\hat D\wedge\hat E,
        &\qquad& \forall\phi\in S_r; \\
  & S_r\otimes\hat D\smile \HH^2(A) = 0; \\
  & 1\otimes\hat E\smile\omega_2 = \omega_3; \\
  & 1\otimes\hat E\smile\kappa\otimes\hat y\wedge\hat D 
        = \kappa\otimes\hat y\wedge\hat E\wedge\hat D,
        && \forall\kappa\in S_{r+1}/\lin{xF_x,xF_y,yF_y}; \\
  & 1\otimes\hat E\smile\psi D\otimes\hat y\wedge\hat D
        = \psi D\otimes\hat y\wedge\hat D\wedge\hat E,
        &&\forall\psi\in S_1; \\
  & 1\otimes\hat E\smile S_r\otimes\hat D\wedge\hat E = 0.
  \end{alignat}
\end{Proposition*}

These equalities completely describe the multiplicative
structure on~$\HH^\bullet(A)$.

\begin{proof}
There is a morphism of complexes
of~$A$-bimodules~$\Delta:\P\to\P\otimes_A\P$ that lifts the canonical
isomorphism $A\to A\otimes_AA$ such that $\Delta=\Delta_K+\Delta_N$, with
\begin{itemize}

\item $\Delta_K:\P\to\P\otimes_A\P$ the map of $A$-bimodules such that
for whenever $p\geq0$ and $v_1$,~\dots,~$v_p\in V$ we have
  \[
  \Delta_K(1|v_1\wedge \cdots\wedge  v_p|1)
    = \sum(-1)^{\epsilon}
      1|v_{i_1}\wedge \cdots\wedge  v_{i_r}|1
        \otimes1|v_{j_1}\wedge \cdots\wedge  v_{j_s}|1,
  \]
with the sum taken over all decompositions $r+s=p$ with $r$,~$s\geq0$, and
all permutations $(i_1,\dots,i_r,j_1,\dots,j_s)$ of $(1,\dots,p)$ such that
$i_1<\cdots<i_r$ and $j_1<\cdots<j_s$, and where $\epsilon$ is the
signature of the permutations,

\item and $\Delta_N:\P\to\P\otimes_A\P$ the map of $A$-bimodules such that
  \begin{gather}
  \Delta_N(1|1) = 0; \\
  \Delta_N(1|v|1) = 0, \qquad \forall v\in V; \\
  \Delta_N(1|v\wedge w|1) = 0, \qquad\text{if $v$,~$w\in\{x,y,D,E\}$, $v\neq
  w$ and $\{v,w\}\neq\{y,D\}$}; \\
  \Delta_N(1|y\wedge D|1)
        = f\pt{1}|f\pt{2}|f\pt{3}\otimes1|f\pt{4}|f\pt{5}; \\
  \Delta_N(1|x\wedge y\wedge D|1) = \Delta_N(1|x\wedge y\wedge E|1)
        = \Delta_N(1|x\wedge D\wedge E|1) = 0; \\
  \!\begin{multlined}[.9\displaywidth]
  \Delta_N(1|y\wedge D\wedge E|1) =
          - f\pt{1}|f\pt{2}\wedge E|f\pt{3}\otimes1|f\pt{4}|f\pt{5} \\
          + f\pt{1}|f\pt{2}|f\pt{3}\otimes1|f\pt{4}\wedge E|f\pt{5}.
  \end{multlined}
  \end{gather}
Here we have written $f\pt{1}|f\pt{2}|f\pt{3}|f\pt{4}|f\pt{5}$
the image of~$F$ under the composition
  \[
  \begin{tikzcd}
  S \arrow[r, "\nabla"]
    & S \otimes S_1\otimes S \arrow[r, "\id_S\otimes\id_{S_1}\otimes\nabla"]
    &[4em] S \otimes S_1\otimes S\otimes S_1\otimes S,
  \end{tikzcd}
  \]
with an omitted sum, \emph{à la} Sweedler.

\end{itemize}
We leave the verification that this does define a morphism of complexes to
the reader.

One can compute the cup product on $\HH^\bullet(A)$ using this diagonal
morphism~$\Delta$. Indeed, we view $\HH^\bullet(A)$ as the cohomology of
the complex~$\hom_{A^e}(\P,A)$, and if $\phi$ and~$\psi$ are a $p$- and a
$q$-cocycle in that complex, the cup product of their cohomology classes is
represented by the composition
  \[
  \begin{tikzcd}
  P_{p+q} \arrow[r, "\Delta_{p,q}"]
    & P_p\otimes_A P_q \arrow[r, "\phi\otimes\psi"]
    & A\otimes_AA=A,
  \end{tikzcd}
  \]
with $\Delta_{p,q}$ the component $P_{p+q}\to P_p\otimes P_q$ of the
morphism~$\Delta$. The multiplication table given in the statement of the
composition can be computed in this way, item by item.
\end{proof}

\begin{Proposition}\label{prop:partial:sub}
\begin{thmlist}\fixhspace

\item For all $i$,~$j$,~$k\in\{0,\dots,r+1\}$ we have
  \[ \label{eq:partial:rels}
  \partial_{\alpha_i}\smile\partial_{\alpha_j}
  + \partial_{\alpha_j}\smile\partial_{\alpha_k}
  + \partial_{\alpha_k}\smile\partial_{\alpha_i} 
  = 0
  \]
and $\HH^1(A)\smile\HH^1(A)=S_r\otimes\hat D\wedge\hat E$. 

\item The subalgebra~$\H$ of~$\HH^\bullet(A)$ generated by~$\HH^1(A)$ is the
graded-commutative algebra freely generated by its elements
$\partial_{\alpha_0}$,~\dots,~$\partial_{\alpha_{r+1}}$ of degree~$1$
subject to the $\binom{r+2}{3}$ relations~\eqref{eq:partial:rels}.

\end{thmlist}
\end{Proposition}

This subalgebra~$\H$ is isomorphic to the de Rham
cohomology of the complement of the arrangement of lines~$\A$. This follows
from a direct computation of this cohomology or, in fact, from the solution
of Arnold's conjecture by Brieskorn; this is discussed in detail in
\cite{OT}*{Section 5.4}.

\begin{proof}
Using Proposition~\pref{prop:cup} and the description given in the proof
of Proposition~\pref{prop:partial} for the
derivations~$\partial_{\alpha_i}$ we compute immediately that
  \[
  \partial_{\alpha_i}\smile\partial_{\alpha_j}
        = - \begin{vmatrix}
            \alpha_{ix} & \alpha_{jx} \\
            \alpha_{iy} & \alpha_{jy} 
            \end{vmatrix}
            \frac{Q}{\alpha_i\alpha_j}
  \]
for all $i$,~$j\in\{0,\dots,r+1\}$. Using this, we see that for all
$i$,~$j$,~$k\in\{0,\dots,r+1\}$ we have
  \[
  \partial_{\alpha_i}\smile\partial_{\alpha_j}
  + \partial_{\alpha_j}\smile\partial_{\alpha_k}
  + \partial_{\alpha_k}\smile\partial_{\alpha_i} 
  =
  - \begin{vmatrix}
    \alpha_i & \alpha_j & \alpha_k \\
    \alpha_{ix} & \alpha_{jx} & \alpha_{kx} \\
    \alpha_{iy} & \alpha_{jy} & \alpha_{ky} \\
    \end{vmatrix}
    \frac{Q}{\alpha_i\alpha_j\alpha_k}
    = 0,
  \]
as the determinant vanishes. This proves the first claim of~\thmitem{1}.
The second one follows immediately from the description of the cup product
of Proposition~\pref{prop:cup}.

\thmitem{2} Let $\mathcal F=\bigoplus_{n\geq0}\mathcal F_n$ be the free graded-commutative
algebra generated by $r+2$ generators $w_0$,~\dots,~$w_{r+1}$ of degree~$1$
subject to the relations $w_iw_j+w_jw_k+w_kw_i=0$, one for each choice of
$i$,~$j$,~$k\in\{0,\dots,r+1\}$. We have $\mathcal F_n=0$ if $n\geq3$: if
$i$,~$j$,~$k\in\{1,\dots,r+1\}$ we have
$w_iw_jw_k=(w_iw_j+w_jw_k+w_kw_i)w_k=0$, because of graded-commutativity.
On the other hand, we have $\dim \mathcal F_2\leq r+1$. To see this, we notice that
$\mathcal F_2$ is spanned by products $w_iw_j$ with $1\leq i<j\leq r+1$. If 
$i+1<j$ then $w_iw_j=-w_{i+1}w_j-w_{i+1}w_i$: it follows from this that the
set of monomials $\{w_iw_{i+1}:0\leq i\leq r\}$ already spans~$\mathcal F_2$.

The first part of the proposition implies that there is a surjective
morphism of graded algebras $f:\mathcal F\to\H$ such that
$f(w_i)=\partial_{\alpha_i}$ for all $i\in\{0,\dots,r+1\}$, and 
this map is also injective because the
dimension of the component of degree~$2$ of~$\H$, which is $S_r\otimes\hat
D\wedge\hat E$, is $r+1$.
\end{proof}

\paragraph\label{p:3-from-2} Proposition~\pref{prop:partial:sub} describes
meaningfully a part of the associative algebra~$\HH^\bullet(A)$, the
subalgebra~$\H$ generated by~$\HH^1(A)$, in terms of the geometry of the
arrangement~$\A$. It is not clear how to make sense of the complete
algebra. We can make the following observation, though. Let us write 
  \[
  \HH^2(A)' = 
        \kk\omega_2\oplus
        (S_{r+1}/\langle xF_x,xF_y,yF_y\rangle\oplus S_1D)\otimes\hat y\wedge\hat D,
   \]
which is a complement of~$\H^2$ in~$\HH^2(A)$, and let
$Q=\alpha_0\dots\alpha_{r+1}$ be a factorization of~$Q$ as a product of
linear factors. If $\delta:A\to A$ is
derivation of~$A$, then our description of~$\HH^1(A)$ implies that there exist 
scalars $\delta_0$,~\dots,~$\delta_{r+1}\in\kk$ and an element $u\in A$
such that $\delta=\sum_{i=0}^{r+1}\delta_i\partial_{\alpha_u}+\ad(u)$, and
it follows easily from Proposition~\pref{prop:cup} that the map
  \[
  \zeta\in\HH^2(A)' \mapsto \delta\smile\zeta\in\HH^3(A)
  \]
is either zero or an isomorphism, provided~$\sum_{i=0}^{r+1}\delta_i$ is
zero or not.

\subsection*{The Gerstenhaber bracket}

\paragraph Using the comparison morphisms of~\pref{p:comparison}, we can
now compute the Gerstenhaber bracket. As usual, this is very
laborious.

\begin{Proposition*}\plabel{prop:bracket}
In $\HH^\bullet(A)$ we have
  \begin{gather}
  [0,\bullet] \quad
      \left\{ \quad
      \begin{lgathered}
      [\HH^0(A),\HH^\bullet(A)] 
          = 0, 
      \end{lgathered}
      \right.
      \\
  [1,1] \quad
      \left\{ \quad
      \begin{lgathered}
      [\HH^1(A),\HH^1(A)] 
          = 0, 
      \end{lgathered}
      \right.
      \\
  [1,2] \quad
      \left\{ \quad
      \begin{lgathered}
      [\HH^1(A),S_r\otimes\hat D\wedge\hat E] 
          = 0, 
          \\
      [u\otimes\hat D+\lambda\otimes\hat E,(v+wD)\otimes\hat y\wedge\hat D]
          = uw\otimes\hat y\wedge\hat D, 
          \\
      [u\otimes\hat D+\lambda\otimes\hat E,\omega_2]
          = ((\mu-\lambda)yF_x+\mu y\bar F-y^2\bar u)\otimes\hat y\wedge\hat D, 
      \end{lgathered}
      \right.
      \\
  [1,3] \quad
      \left\{ \quad
      \begin{lgathered}
      [u\otimes\hat D+\lambda\otimes\hat E,(v+wD)\otimes\hat y\wedge\hat D\wedge\hat E]
          = uw\otimes\hat y\wedge\hat D\wedge\hat E, \\
      [u\otimes\hat D+\lambda\otimes\hat E,\omega_3]
          = ((\mu-\lambda)yF_x+\mu y\bar F-y^2\bar u)
                \otimes\hat y\wedge\hat D\wedge\hat E, 
      \end{lgathered}
      \right.
      \\
  [2,2] \quad
      \left\{ \quad
      \begin{lgathered}
      [S_r\otimes\hat D\wedge\hat E,S_r\otimes\hat D\wedge\hat E] 
          = 0, 
          \\
      [u\otimes\hat D\wedge\hat E,(v+wD)\otimes\hat y\wedge\hat D]
          = uw\otimes\hat y\wedge\hat D\wedge\hat E, 
          \\
      [u\otimes\hat D\wedge\hat E,\omega_2] 
          = (\mu yFx+\mu y\bar F-y^2\bar u)\otimes\hat y\wedge\hat D\wedge\hat E, 
          \\
      [(S_{r+1}+S_1D)\otimes\hat y\wedge\hat D,(S_{r+1}+S_1D)\otimes\hat y\wedge\hat D]
          = 0,
          \\
      [(S_{r+1}+S_1D)\otimes\hat y\wedge\hat D,\omega_2]
          = 0,
          \\
      [\omega_2,\omega_2]
          = 0.
      \end{lgathered}
      \right.
  \end{gather}
Here $u\in S_r$, $\lambda\in\kk$, $v\in S_{r+1}$, $w\in S_1$ and
$\mu\in\kk$ and $\bar u\in S_{r-1}$ are such that $u=\lambda y^r+x\bar u$.
\end{Proposition*}

\begin{proof}
Let us first recall from~\cite{Gerstenhaber} how one can compute the
Gerstenhaber bracket in the standard complex $\hom_{A^e}(\B A,A)$. If
$f:A^{\otimes q}\to A$ is a $q$-cochain in the standard complex
$\hom_{A^e}(\B A,A)$, which we identify as usual with
$\hom(A^{\otimes\bullet},A)$, and $p\geq q$, we denote $\w_p(f):A^{\otimes
p}\to A^{p-q+1}$ the $p$-cochain in the same complex such that
  \begin{multline}
  \w_p(f)(a_1\otimes\cdots\otimes a_{p}) \\
    = \sum_{i=1}^{p-q+1}(-1)^{(q-1)(i-1)}
        a_1\otimes\cdots\otimes a_{i-1}
        \otimes f(a_i\otimes\cdots\otimes a_{i+q-1})
        \otimes a_{i+q}\otimes\cdots\otimes a_{p}.
  \end{multline}
If now $\alpha$ and $\beta$ are a $p$- and a $q$-cocycle in the standard
complex, the Gerstenhaber composition~$\diamond$ (which is usually written
simply~$\circ$) of $\alpha$ and~$\beta$ is the $(p+q-1)$-cochain
  \[
  \alpha\diamond\beta = \alpha\circ\w_{p+q-1}(\beta)
  \]
and the Gerstenhaber bracket is the graded commutator for this composition,
so that
  \[
  [\alpha,\beta] = \alpha\diamond\beta
    -(-1)^{(p-1)(q-1)}  \beta\diamond\alpha.
  \]
Next, if $\alpha$ and $\beta$ are now a $p$- and a~$q$-cochain 
in the complex $\hom_{A^e}(\P,A)$, we
can lift them to a $p$-cochain $\tilde\alpha=\alpha\circ\psi_p$ and a
$q$-cochain $\tilde\beta=\beta\circ\psi_q$ in the
standard complex $\hom_{A^e}(\B A,A)$, and the Gerstenhaber bracket of the
classes of~$\alpha$ and~$\beta$ is then represented by the
$(p+q-1)$-cochain $[\tilde\alpha,\tilde\beta]\circ\phi_{p+q-1}$. This is
the computation we have to do in order to compute brackets
in~$\HH^\bullet(A)$, except that in some favorable circumstances we can
take advantage of the compatibility of the bracket with the product to cut
down the work. We do this in several steps.
\begin{itemize}

\item Since the morphism $\psi$ is normalized and $\HH^0(A)$ is spanned by~$1\in\kk$, it
follows immediately that 
  \[
  [\HH^0(A),\HH^\bullet(A)]=0.
  \]

\item The Gerstenhaber bracket on~$\HH^1(A)$ is induced by the commutator
of derivations. From Proposition~\pref{prop:partial} we have a basis
of~$\HH^1(A)$ whose elements are classes of certain derivations, and it is
immediate to check that those derivations commute, so that 
  \[ \label{eq:br-11}
  [\HH^1(A),\HH^1(A)]=0.
  \]

\item We know that the subspace~$S_r\otimes\hat D\wedge\hat E$ of~$\HH^2(A)$ is
$\HH^1(A)\smile\HH^1(A)$. Since $\HH^\bullet(A)$ is a Gerstenhaber algebra
and we now that~\eqref{eq:br-11} holds, it follows that
  \[
  [\HH^1(A),S_r\otimes\hat D\wedge\hat E] = 0.
  \]
For exactly the same reasons we also have that
  \[
  [S_r\otimes\hat D\wedge\hat E,S_r\otimes\hat D\wedge\hat E]=0.
  \]

\item Let $\alpha=u\otimes\hat D+\lambda\otimes\hat E$, with $u\in S_r$ and
$\lambda\in\kk$. If $\beta=(v+wD)\otimes\hat y\wedge\hat D$, with $v\in
S_{r+1}$ and $w\in S_1$, one can compute that
$(\tilde\alpha\diamond\tilde\beta)\circ\phi=uw\otimes\hat y\wedge\hat D$
and that $(\tilde\beta\diamond\tilde\alpha)\circ\phi=0$: it follows from
this that
  \[
  [\alpha,(v+wD)\otimes\hat y\wedge\hat D] = uw\otimes\hat y\wedge\hat D.
  \]
On the other hand, we have
$(\tilde\omega_2\diamond\tilde\alpha)\circ\phi=0$ and
  \begin{align}
  [\tilde\alpha,\tilde\omega_2]\circ\phi
    &= (\tilde\alpha\diamond\tilde\omega_2)\circ\phi
     = (yu-\lambda y^{r+1})\otimes\hat x\wedge\hat D
       +\lambda y\bar F\otimes\hat y\wedge\hat D \\
    &= \!\begin{multlined}[t][.75\displaywidth]
        \bigl(
        (\mu-\lambda)yF_x+\mu y \bar F-y^2\bar u
       \bigr)\otimes\hat y\wedge\hat D \\
       - \delta^1\bigl(
                 ((\mu-\lambda)\bar F-y\bar u)E\otimes\hat D
                 + (\lambda-\mu)y\otimes\hat x
                 \bigr)
       \end{multlined}
  \end{align}
with $\bar u\in S_{r-1}$ and $\mu\in\kk$ chosen so that $u=\mu y^r+x\bar u$.

Finally, if $v\in S_{r+1}$ and $w\in S_1$, using the compatibility
of the bracket and the product and what we know so far we see that
  \begin{align}
  [\alpha,(v+wD)\otimes\hat y\wedge \hat D\wedge\hat E]
    &= [\alpha,1\otimes E\smile (v+wD)\otimes\hat y\wedge \hat D] \\
    &= 1\otimes E\smile [\alpha,(v+wD)\otimes\hat y\wedge \hat D] \\
    &= 1\otimes E\smile uw\otimes\hat y\wedge\hat D \\
    &= uw\otimes\hat y\wedge\hat D\wedge\hat E 
  \end{align}
and, similarly, that
  \begin{align}
  [\alpha,\omega_3]
    &= [\alpha,\omega_2\smile1\otimes\hat E]
     = [\alpha,\omega_2]\smile1\otimes\hat E
       + \omega_2\smile[\alpha,1\otimes\hat E] \\
    &= \bigl((\mu-\lambda)yF_x+\mu y \bar F-y^2\bar u \bigr)
         \otimes\hat y\wedge\hat D\wedge\hat E.
  \end{align}

\item Let $u\in S_r$. If $v\in S_{r+1}$ and $w\in S_1$, we have
  \begin{align}
  \MoveEqLeft{}
  [u\otimes\hat D\wedge\hat E,(v+wD)\otimes\hat y\wedge\hat D]
      = [u\otimes\hat D\smile1\otimes\hat E,(v+wD)\otimes\hat y\wedge\hat D] \\
    & = \!\begin{multlined}[t][.8\displaywidth]
        [u\otimes\hat D,(v+wD)\otimes\hat y\wedge\hat D]\smile1\otimes\hat E \\
        +
        u\otimes\hat D\smile[1\otimes\hat E,(v+wD)\otimes\hat y\wedge\hat D] 
        \end{multlined}
        \\
    & = uw\otimes\hat y\wedge\hat D\smile 1\otimes\hat E
      = uw\otimes\hat y\wedge\hat D\wedge\hat E.
  \end{align}
Similarly, 
  \begin{align}
  [u\otimes\hat D\wedge\hat E,\omega_2]
    & = [u\otimes\hat D\smile1\otimes\hat E,\omega_2] \\
    & = [u\otimes\hat D,\omega_2]\smile1\otimes\hat E
        + u\otimes\hat D\smile[1\otimes\hat E,\omega_2] \\
    & = \bigl(
        \mu yF_x+\mu y \bar F-y^2\bar u
       \bigr)\otimes\hat y\wedge\hat D\wedge\hat E.
  \end{align}
if $u=\mu y^r+x\bar u$ with $\mu\in\kk$ and $\bar u\in S_{r-1}$.

\item Let now $\alpha=(v+wD)\otimes\hat y\wedge\hat D$ and
$\beta=(s+tD)\otimes\hat y\wedge\hat D$, with $v$,~$s\in S_{r+1}$ and
$w$,~$t\in S_1$. We claim that
$(\tilde\alpha\diamond\tilde\beta)\circ\phi=0$, so that, by symmetry, we
have $[\tilde\alpha,\tilde\beta]\circ\phi=0$. To verify our claim, we
compute:
{\def\fix#1{\hphantom{1|x\wedge D\wedge E|1}\mathllap{#1}}
  \begin{align}
  \fix{1|x\wedge y\wedge E|1}
    & \xmapsto{\phi} \something{\kk[x,y,E]^{\otimes 5}} 
      \xmapsto{\w_3(\tilde\beta)} 0; \\
  \fix{1|x\wedge D\wedge E|1}
    & \xmapsto{\phi} \something{\kk[x,D,E]^{\otimes 5}} 
      \xmapsto{\w_3(\tilde\beta)} 0; \\
  \fix{1|x\wedge y\wedge D|1}
    & \xmapsto{\phi}
      \!\begin{multlined}[t][.7\displaywidth]
      1|x|y|D|1 
      - 1|x|D|y|1 
      + 1|D|x|y|1 \\
      - 1|D|y|x|1 
      + 1|y|D|x|1 
      - 1|y|x|D|1 
      + \something{S^{\otimes 5}}
      \end{multlined}
      \\
      & \xmapsto{\w_3(\tilde\beta)}
      1|(s+tD)|x|1-1|x|(s+tD)|1
      \\
      & \xmapsto{\psi}
        - s\pt{1}|x\wedge s\pt{2}|s\pt{3}
        - t\pt{1}|x\wedge t\pt{2}|t\pt{3}D
        - t|x\wedge D|1
      \\
      & \xmapsto{\alpha}
        0;
      \\
  \fix{1|y\wedge D\wedge E|1}
      & \xmapsto{\phi}
      \!\begin{multlined}[t][.7\displaywidth]
      1|y|D|E|1 
      - 1|y|E|D|1 
      + 1|E|y|D|1 \\
      - 1|E|D|y|1 
      + 1|D|E|y|1 
      - 1|D|y|E|1 
      + \something{\kk[x,y,E]^{\otimes 5}}
      \end{multlined}
      \\
      & \xmapsto{\w_3(\tilde\beta)}
      1|(s+tD)|E|1 - 1|E|(s+tD)|1 
      \\
      & \xmapsto{\psi}
      s\pt1|s\pt2\wedge E|s\pt3 + t\pt1|t\pt2\wedge E|t\pt3D + t|D\wedge E|1
      \\
      & \xmapsto{\alpha}
      0.
  \end{align}}

\item Let again $\alpha=(v+wD)\otimes\hat y\wedge\hat D$, with $v\in
S_{r+1}$ and $w\in S_1$, and let us compute that
$(\tilde\omega_2\diamond\tilde\alpha)\circ\phi_3=-w(yD-y^{r+1}E)\otimes\hat
x\wedge\hat y\wedge\hat D$.
  \begin{align*}
  1|x\wedge y\wedge z|1
     & \xmapsto{\phi_3}
     \something{\kk[x,y,E]^{\otimes5}}
     \xmapsto{\w_2(\tilde\alpha)}
     0 \\
   1|x\wedge D\wedge E|1
     & \xmapsto{\phi_3}
     \something{\kk[x,D,E]^{\otimes 5}}
     \xmapsto{\w_3(\tilde\alpha)}
     0 \\
  1|x\wedge y\wedge D|1
    & \xmapsto{\phi_3}
      \!\begin{multlined}[t][.7\displaywidth]
      1|x|y|D|1 
      - 1|x|D|y|1 
      + 1|D|x|y|1 \\
      - 1|D|y|x|1 
      + 1|y|D|x|1 
      - 1|y|x|D|1 
      + \something{S^{\otimes 5}}
      \end{multlined}
      \\
    & \xmapsto{\w_3(\tilde\alpha)}
      1|(v+wD)|x|1 + 1|x|(v+wD)|1
      \\
    & \xmapsto{\psi_2}
      -v\pt1|x\wedge v\pt 2|v\pt3 
      -w\pt1|x\wedge w\pt 2|w\pt3 D
      -w|x\wedge D|1
      \\
    & \xmapsto{\omega_2}
      -w(yD-y^{r+1}E)
      \\
  1|y\wedge D\wedge E|1
    & \xmapsto{\phi_3}
    \!\begin{multlined}[t][.7\displaywidth]
      1|y|D|E|1 
      - 1|y|E|D|1 
      + 1|E|y|D|1 \\
      - 1|E|D|y|1 
      + 1|D|E|y|1 
      - 1|D|y|E|1 
      + \something{\kk[x,y,E]^{\otimes 5}}
      \end{multlined}
      \\
      & \xmapsto{\w_3(\tilde\alpha)}
      1|(v+wD)|E|1 - 1|E|(v+wD)|1 
      \\
      & \xmapsto{\psi_2}
      v\pt1|v\pt2\wedge E|v\pt3 + w\pt1|w\pt2\wedge E|w\pt3D + w|D\wedge E|1
      \\
      & \xmapsto{\omega_2}
      0.
  \end{align*}
Similarly, we have that
$(\tilde\alpha\diamond\tilde\omega_2)\circ\phi_3=y(v+wD)\otimes \hat x\wedge
\hat y\wedge\hat D$:
  \begin{align*}
  1|x\wedge y\wedge z|1
     & \xmapsto{\phi_3}
       \something{\kk[x,y,E]^{\otimes5}}
       \\
     & \xmapsto{\w_2(\tilde\omega_2)}
       0 \\
   1|x\wedge D\wedge E|1
     & \!\begin{multlined}[t][.7\displaywidth]
       \xmapsto{\phi_3}
       1|x|D|E|1 
       - 1|x|E|D|1 
       + 1|E|x|D|1 \\
       - 1|E|D|x|1 
       + 1|D|E|x|1 
       - 1|D|x|E|1 
       \end{multlined}
       \\
     & \xmapsto{\w_3(\tilde\omega_2)}
       -1|E|(yD-y^{r+1}E)|1+1|(yD-y^{r+1})|E|1
       \\
     & \xmapsto{\psi_2}
       -1|y\wedge E|D-y|D\wedge E|1+\sum_{i=0}^ry^i|y\wedge E|y^{r-i}
       \\
     & \xmapsto{\alpha}
       0
       \\
  1|x\wedge y\wedge D|1
    & \!\begin{multlined}[t][.7\displaywidth]
      \xmapsto{\phi_3}
      1|x|y|D|1 
      - 1|x|D|y|1 
      + 1|D|x|y|1 \\
      - 1|D|y|x|1 
      + 1|y|D|x|1 
      - 1|y|x|D|1 
      + \something{S^{\otimes 5}}
      \end{multlined}
      \\
    & \!\begin{multlined}[t][.7\displaywidth]
      \xmapsto{\w_3(\tilde\omega_2)}
      -1|x|y\bar F E|x|1 - 1|(yD-y^{r+1}E)|y|1 \\
      +1|y\bar FE|x|1+1|y|(yD-y^{r+1}E)|1
      \end{multlined}
      \\
    & \!\begin{multlined}[t][.7\displaywidth]
      \xmapsto{\psi_2}
      y|y\wedge D|1-y^{r+1}|y\wedge E|1 \\
      -(y\bar FE)\pt1|x\wedge(y\bar FE)\pt2|(y\bar FE)\pt3
      \end{multlined}
      \\
    & \xmapsto{\alpha}
      -y(v+wD)
      \\
  1|y\wedge D\wedge E|1
    & \!\begin{multlined}[t][.7\displaywidth]
      \xmapsto{\phi_3}
      1|y|D|E|1 
      - 1|y|E|D|1 
      + 1|E|y|D|1 \\
      - 1|E|D|y|1 
      + 1|D|E|y|1 
      - 1|D|y|E|1 
      + \something{\kk[x,y,E]^{\otimes 5}}
      \end{multlined}
      \\
      & \xmapsto{\w_3(\tilde\omega_2)}
      -1|E|y\bar FE|1 + 1|y\bar FE|E|1 
      \\
      & \xmapsto{\psi_2}
      (y\bar FE)\pt1|(y\bar FE)\pt2\wedge E|(y\bar FE)\pt3
      \\
      & \xmapsto{\alpha}
      0.
  \end{align*}
It follows from this that
  \begin{align}
  [\tilde\omega_2,\tilde\alpha]\circ\phi_3
     &= -w (yD-y^{r+1}E)\otimes\hat x\wedge\hat y\wedge\hat D
        + y(v+wD)\otimes \hat x\wedge \hat y\wedge\hat D 
        \\
     &= (yv+y^{r+1}E)\otimes\hat x\wedge\hat y\wedge\hat D
  \end{align}
and, as we say in~\pref{p:3-cocycles}, this is a coboundary.

\item The one computation that remains is that of the bracket of~$\omega_2$
with itself, which is represented by the $3$-cocycle
  \[ \label{eq:w2w2}
  [\tilde\omega_2,\tilde\omega_2]\circ\phi_3
    = 2(\tilde\omega_2\diamond\tilde\omega_2)\circ\phi_3
    = 2y^2\bar FE\otimes\hat x\wedge\hat y\wedge\hat D,
  \]
as can be seen from the following calculation:
  \begin{align*}
  1|x\wedge y\wedge z|1
     & \xmapsto{\phi_3}
       \something{\kk[x,y,E]^{\otimes5}}
       \\
     & \xmapsto{\w_2(\tilde\omega_2)}
       0 \\
   1|x\wedge D\wedge E|1
     & \!\begin{multlined}[t][.7\displaywidth]
       \xmapsto{\phi_3}
       1|x|D|E|1 
       - 1|x|E|D|1 
       + 1|E|x|D|1 \\
       - 1|E|D|x|1 
       + 1|D|E|x|1 
       - 1|D|x|E|1 
       \end{multlined}
       \\
     & \xmapsto{\w_3(\tilde\omega_2)}
       -1|E|(yD-y^{r+1}E)|1+1|(yD-y^{r+1})|E|1
       \\
     & \xmapsto{\psi_2}
       -1|y\wedge E|D-y|D\wedge E|1+\sum_{i=0}^ry^i|y\wedge E|y^{r-i}
       \\
     & \xmapsto{\omega_2}
       0
       \\
  1|x\wedge y\wedge D|1
    & \!\begin{multlined}[t][.7\displaywidth]
      \xmapsto{\phi_3}
      1|x|y|D|1 
      - 1|x|D|y|1 
      + 1|D|x|y|1 \\
      - 1|D|y|x|1 
      + 1|y|D|x|1 
      - 1|y|x|D|1 
      + \something{S^{\otimes 5}}
      \end{multlined}
      \\
    & \!\begin{multlined}[t][.7\displaywidth]
      \xmapsto{\w_3(\tilde\omega_2)}
      -1|x|y\bar F E|x|1 - 1|(yD-y^{r+1}E)|y|1 \\
      +1|y\bar FE|x|1+1|y|(yD-y^{r+1}E)|1
      \end{multlined}
      \\
    & \!\begin{multlined}[t][.7\displaywidth]
      \xmapsto{\psi_2}
      y|y\wedge D|1-y^{r+1}|y\wedge E|1 \\
      -(y\bar FE)\pt1|x\wedge(y\bar FE)\pt2|(y\bar FE)\pt3
      \end{multlined}
      \\
    & \xmapsto{\omega_2}
      -y^2\bar FE
      \\
  1|y\wedge D\wedge E|1
    & \!\begin{multlined}[t][.7\displaywidth]
      \xmapsto{\phi_3}
      1|y|D|E|1 
      - 1|y|E|D|1 
      + 1|E|y|D|1 \\
      - 1|E|D|y|1 
      + 1|D|E|y|1 
      - 1|D|y|E|1 
      + \something{\kk[x,y,E]^{\otimes 5}}
      \end{multlined}
      \\
      & \xmapsto{\w_3(\tilde\omega_2)}
      -1|E|y\bar FE|1 + 1|y\bar FE|E|1 
      \\
      & \xmapsto{\psi_2}
      (y\bar FE)\pt1|(y\bar FE)\pt2\wedge E|(y\bar FE)\pt3
      \\
      & \xmapsto{\omega_2}
      0.
  \end{align*}
Now the $3$-cocycle~\eqref{eq:w2w2} is a coboundary, again by what we saw
in~\pref{p:3-cocycles}, so that the class of~$\omega_2$ has bracket-square
zero.
\end{itemize}
This completes the proof of the proposition.
\end{proof}

\section{Hochschild homology, cyclic homology and \texorpdfstring{$K$}{K}-theory}
\label{sect:rest}

\paragraph For completeness, we determine the rest of the `usual'
homological invariants of our algebra $A$. Recall that our ground
field~$\kk$ is of characteristic zero.

\begin{Proposition*}\plabel{prop:hom}
The inclusion $T=\kk[E]\to A$ induces an isomorphism in Hochschild
homology and in cyclic homology. In particular, there are isomorphisms of
vector spaces
  \begin{align}
  &  \HH_i(A) \cong
        \begin{cases*}
        T, & \text{if $i=0$ or $i=1$;} \\
        0, & \text{if $i\geq2$;}
        \end{cases*}
  && \HC_i(A) \cong
        \begin{cases*}
        T, & \text{if $i=0$;} \\
        \HC_i(\kk), & \text{if $i>0$.}
        \end{cases*}
  \end{align}
On the other hand, the inclusion~$\kk\to A$ induces an
isomorphism in periodic cyclic homology and in higher $K$-theory.
\end{Proposition*}

\begin{proof}
As we know, the algebra~$A$ is $\NN_0$-graded and for each $n\in\NN_0$ its
homogeneous component~$A_n$ of degree~$n$ is the eigenspace corresponding
to the eigenvalue~$n$ of the derivation~$\ad(E):A\to A$. On one hand, this
grading of~$A$ induces as usual an $\NN_0$-grading on the Hochschild
homology~$\HH_\bullet(A)$ of~$A$; on the other, the derivation~$\ad(E)$
induces a linear map $L_{\ad(E)}:\HH_\bullet(A)\to\HH_\bullet(A)$ as
in~\cite{Loday}*{\oldS4.1.4} and, in fact, for all $n\in\NN_0$ the
homogeneous component $\HH_\bullet(A)_n$ of degree~$n$ for that grading
coincides with the eigenspace corresponding to the eigenvalue~$n$
of~$L_{\ad(E)}$. As the derivation~$\ad(E)$ is inner, it follows from
\cite{Loday}*{Proposition 4.1.5} that the map~$L_{\ad(E)}$ is actually the
zero map and this tells us in our situation that $\HH_\bullet(A)_n=0$ for
all $n\neq0$. Of course, this means that $\HH_\bullet(A)=\HH_\bullet(A)_0$
and, since $A$ is non-negatively graded, it is immediate that the 0th
homogeneous component $\HH_\bullet(A)_0$ coincides with the Hochschild
homology~$\HH_\bullet(A_0)$ of $A_0$ and that the map
$\HH_\bullet(A_0)\to\HH_\bullet(A)$ induced by the inclusion
$A_0\hookrightarrow A$ is an isomorphism.
Now, in the notation of~\cite{Loday}*{Theorem 4.1.13}, this tells us that
$\overset\approx\HH_\bullet(A)=0$ so that by that theorem we also have
$\overset\approx\HC_\bullet(A)=0$: this means precisely that the inclusion
$A_0\hookrightarrow A$ induces an isomorphism
$\HC_\bullet(A_0)\to\HC_\bullet(A)$ in cyclic homology. Together with the
well-known computation of the Hochschild homology of a polynomial
ring and that of the cyclic homology of symmetric algebras
\cite{Loday}*{Theorem 3.2.5}, this proves the first claim of the statement.

In the proof of the lemma of~\pref{p:P} we constructed an increasing
filtration~$F$ on the algebra~$A$ with $F_{-1}A=0$ and such that the
corresponding graded algebra is the commutative polynomial ring $\gr
A=\kk[x,y,D,E]$ with generators $x$ and~$y$ in degree~$0$ and $D$ and~$E$
in degree~$1$. In particular, both $\gr A$ and its subalgebra~$\gr_0A$ of
degree~$0$ have finite global dimension. It follows from a theorem of
D.\,Quillen \cite{Quillen}*{p.\,117, Theorem~7} that the inclusion
$\kk[x,y]=F_0A\to A$ induces an isomorphism $K_i(\kk[x,y])\to K_i(A)$ in
$K$-theory for all $i\geq0$. Similarly, the theorem of J.\,Block
\cite{Block}*{Theorem 3.4} tells us that that inclusion induces an
isomorphism $HP_\bullet(\kk[x,y])\to HP_\bullet(A)$ in periodic cyclic
homology. As the inclusion $\kk\to\kk[x,y]$ induces an isomorphism in
$K$-theory and in periodic cyclic homology, we see that the second claim of
the proposition holds.
\end{proof}

\section{The Calabi--Yau property}
\label{sect:cy}

\paragraph\label{p:vee} The enveloping algebra~$A^e$ of~$A$ is a bimodule over
itself, with left and right actions $\triangleright$ and $\triangleleft$
given by `outer' and `inner' multiplication, respectively, so that if
$a\otimes b$, $c\otimes d$ and $e\otimes f$ are elementary tensors
in~$A^e$, we have 
  \[
  a\otimes b \triangleright c\otimes d \triangleleft e\otimes f
    = ace\otimes fdb.
  \]
From this bimodule structure we obtain a duality functor
  \[
  \hom_{A^e}(\place,A^e) : \lmod{A^e}\to\rmod{A^e}.
  \]
On the other hand, using the anti-automorphism $\tau:A^e\to A^e$
such that $\tau(a\otimes b)=b\otimes a$ for all $a$,~$b\in A$, we can turn
a right $A^e$-module $M$ into a left $A^e$-module, with action
$u\triangleright m=m\triangleleft \tau(u)$ for all $u\in A^e$ and all $m\in
M$. In this way, we obtain an isomorphism of categories
$\tau^*:\rmod{A^e}\to\lmod{A^e}$. We denote
$(\place)^\vee:\lmod{A^e}\to\lmod{A^e}$ the composition
$\tau^*\circ\hom_{A^e}(\place,A^e)$.

Let now $W$ be a finite dimensional vector space,
let $W^*$ be the vector space dual to~$W$, and view $A\otimes W\otimes A$
and $A\otimes W^*\otimes A$ as left $A^e$-modules using the usual
`exterior' action. There is a unique $\kk$-linear map
  \[
  \Phi:A\otimes W^*\otimes A\to (A\otimes W\otimes A)^\vee
  \]
such that $\Phi(a\otimes\phi\otimes b)(1\otimes w\otimes 1)=\phi(w)b\otimes
a$ and it is an isomorphism of left $A^e$-modules: we will view it in all
that follows as an identification.

\begin{Proposition}\label{prop:CY}
The algebra $A$ is twisted Calabi-Yau of dimension~$4$ with modular
automorphism $\sigma:A\to A$ such that
  \begin{align}
  &  \sigma(x) = x,
  && \sigma(y) = y,
  && \sigma(D) = D + F_y,
  && \sigma(E) = E + r+2.
  \end{align}
\end{Proposition}

Let us recall from~\cite{Ginzburg} that this means that $A$ has a
resolution of finite length by finitely generated projective $A$-bimodules,
that $\Ext_{A^e}^i(A,A^e)=0$ if $i\neq4$ and that $\Ext_{A^e}^4(A,A^e)\cong
A_\sigma$, the $A$-bimodule obtained from $A$ by twisting its right action
using the automorphism~$\sigma$, so that $a\triangleright x\triangleleft
b=ax\sigma(b)$ for all $a$,~$b\in A$ and all $x\in A_\sigma$.

\begin{proof}
A direct computation shows that there is indeed an automorphism~$\sigma$
of~$A$ as in the statement of the proposition.
We already know that $A$ has a resolution~$\P$ of length~$4$ by finitely
generated free $A$-bimodules, so we need only
compute~$\Ext_{A^e}^\bullet(A,A^e)$, and this is the cohomology of the
complex~$\P^\vee$ obtained by applying the functor described in~\pref{p:vee}
to~$\P$. Using the identifications introduced there, this
complex~$\P^\vee$ is
  \[
  \begin{tikzcd}[column sep=0.95em]
  A\otimes A \arrow[r, "d_1^\vee"]
        & A\otimes V^*\otimes A \arrow[r, "d_2^\vee"]
        & A\otimes \Lambda^2V^*\otimes A \arrow[r, "d_3^\vee"]
        & A\otimes \Lambda^3V^*\otimes A \arrow[r, "d_4^\vee"]
        & A\otimes \Lambda^4V^*\otimes A
  \end{tikzcd}
  \]
with left $A^e$-linear differentials such that
  \begin{gather*}
  d_1^\vee(1\otimes 1) 
        = - [x,1\otimes \xx\otimes 1] 
          - [y,1\otimes \yy\otimes 1] 
          - [D,1\otimes \DD\otimes 1] 
          - [E, 1\otimes \EE\otimes 1];
  \\
  \!\begin{multlined}[.9\displaywidth]
  d_2^\vee( 1 \otimes \xx \otimes  1) 
        =   [y, 1\otimes  \xx\wedge\yy \otimes 1 ] 
          + [D, 1\otimes  \xx\wedge\DD \otimes 1 ]
          + [E, 1\otimes  \xx\wedge\EE \otimes 1 ] \\
          + 1\otimes \xx\wedge\EE\otimes  1  
          + \tilde\nabla^{\yy\wedge\DD}_x(F) ;
  \end{multlined}
  \\
  \!\begin{multlined}[.9\displaywidth]
  d_2^\vee( 1 \otimes \yy \otimes  1) 
        = - [x, 1\otimes  \xx\wedge\yy \otimes 1 ] 
          + [D, 1\otimes  \yy\wedge\DD \otimes 1 ] 
          + [E, 1\otimes  \yy\wedge\EE \otimes 1 ] \\
          + 1\otimes \yy\wedge\EE\otimes  1
          + \tilde\nabla^{\yy\wedge\DD}_y(F) ;
  \end{multlined}
  \\
  \!\begin{multlined}[.9\displaywidth]
  d_2^\vee( 1 \otimes \DD \otimes  1) 
        = - [x, 1\otimes  \xx\wedge\DD \otimes 1 ] 
          - [y, 1\otimes  \yy\wedge\DD \otimes 1 ]
          + [E, 1\otimes  \DD\wedge\EE \otimes 1 ] \\
          + r\otimes \DD\wedge\EE\otimes  1 ;
  \end{multlined}
  \\
  d_2^\vee( 1 \otimes \EE \otimes  1) 
        = - [x, 1\otimes  \xx\wedge\EE \otimes 1 ] 
          - [y, 1\otimes  \yy\wedge\EE \otimes 1 ]
          - [D, 1\otimes  \DD\wedge\EE \otimes 1 ] ;
  \\
  \!\begin{multlined}[.9\displaywidth]
    d_3^\vee(1 \otimes  \xx\wedge\yy\otimes  1) 
        = - [D, 1\otimes \xx\wedge\yy\wedge\DD\otimes 1]
          - \tilde\nabla^{\xx\wedge\yy\wedge\DD}_y(F) \\
          - [E, 1\otimes \xx\wedge\yy\wedge\EE\otimes 1] 
          - 2\otimes \xx\wedge\yy\wedge\EE\otimes 1 ;
  \end{multlined}
  \\
  \!\begin{multlined}[.9\displaywidth]
  d_3^\vee(1 \otimes  \xx\wedge\DD\otimes  1) 
        =   [y, 1\otimes \xx\wedge\yy\wedge\DD\otimes 1]
          - [E, 1\otimes \xx\wedge\DD\wedge\EE\otimes 1] \\
          - (r+1)\otimes \xx\wedge\DD\wedge\EE\otimes 1 ;
  \end{multlined}
  \\
  \!\begin{multlined}[.9\displaywidth]
  d_3^\vee(1 \otimes  \xx\wedge\EE\otimes  1) 
        =   [y, 1\otimes \xx\wedge\yy\wedge\EE\otimes 1]
          + [D, 1\otimes \xx\wedge\DD\wedge\EE\otimes 1] \\
          + \tilde\nabla^{\yy\wedge\DD\wedge\EE}_x(F) ;
  \end{multlined}
  \\
  \!\begin{multlined}[.9\displaywidth]
  d_3^\vee(1 \otimes  \yy\wedge\DD\otimes  1) 
        = - [x, 1\otimes \xx\wedge\yy\wedge\DD\otimes 1]
          - [E, 1\otimes \yy\wedge\DD\wedge\EE\otimes 1] \\
          - (r+1)\otimes \yy\wedge\DD\wedge\EE\otimes 1 ;
  \end{multlined}
  \\
  \!\begin{multlined}[.9\displaywidth]
  d_3^\vee(1 \otimes  \yy\wedge\EE\otimes  1) 
        = - [x, 1\otimes \xx\wedge\yy\wedge\EE\otimes 1]
          + [D, 1\otimes \yy\wedge\DD\wedge\EE\otimes 1] \\
          + \tilde\nabla^{\yy\wedge\DD\wedge\EE}_y(F) ;
  \end{multlined}
  \\
  d_3^\vee(1\otimes  \DD\wedge\EE\otimes  1) 
        = - [x, 1\otimes \xx\wedge\DD\wedge\EE\otimes 1] 
          - [y,1\otimes \yy\wedge\DD\wedge\EE\otimes 1] ;
  \\
  \!\begin{multlined}[.9\displaywidth]
  d_4^\vee (1 \otimes  \xx\wedge\yy\wedge\DD\otimes 1) 
        =   [E , 1 \otimes \xx\wedge\yy\wedge\DD\wedge\EE\otimes 1] \\
          + (r+2) \otimes \xx\wedge\yy\wedge\DD\wedge\EE\otimes 1;
  \end{multlined}
  \\
  d_4^\vee (1 \otimes  \xx\wedge\yy\wedge\EE\otimes 1) 
        = - [D , 1 \otimes \xx\wedge\yy\wedge\DD\wedge\EE\otimes 1]
          - \tilde\nabla^{\xx\wedge\yy\wedge\DD\wedge\EE}_y(F);
  \\
  d_4^\vee( 1\otimes  \xx\wedge\DD\wedge\EE\otimes 1) 
        =   [y, 1\otimes \xx\wedge\yy\wedge\DD\wedge\EE\otimes 1] ;
  \\
  d_4^\vee ( 1\otimes \yy\wedge\DD\wedge\EE\otimes 1) 
        = - [x, 1 \otimes \xx\wedge\yy\wedge\DD\wedge\EE\otimes 1],
  \end{gather*}
where each $\tilde\nabla^u_x$ is the image of $\nabla^u_x$ under the map
$a\otimes u\otimes b\mapsto b\otimes u\otimes a$, and the same with each
$\tilde\nabla^u_y$.

Let us now identify $\P\otimes_A A_\sigma $ with $\P$ as vector spaces,
remembering that the bimodule structure on $\P$ with this identification is
given by $a\triangleright x\triangleleft b=ax\sigma(b)$ for all~$a$,~$b\in
A$ and all~$x\in \P$. There is a morphism of complexes of $A$-bimodules
$\psi : \P^\vee\to \P\otimes_AA_\sigma$ such that
  \begin{gather}
  \psi(1\otimes  \xx\wedge\yy\wedge\DD\wedge\EE \otimes 1) 
        = 1\otimes 1 ; \\[3pt]
  \psi(1\otimes  \yy\wedge\DD\wedge\EE \otimes 1)
        = -1\otimes x\otimes 1 ; \\
  \psi( 1\otimes \xx\wedge\DD\wedge\EE\otimes 1)
      = 1\otimes y\otimes 1 ; \\
  \psi( 1\otimes  \xx\wedge\yy\wedge\EE \otimes  1)
      = -1\otimes D\otimes 1 - \xi ; \\
  \psi( 1\otimes  \xx\wedge\yy\wedge\DD \otimes 1)
      =  1 \otimes E\otimes  1 ; \\[3pt]
  \psi( 1\otimes  \DD\wedge\EE \otimes 1)
      = -1 \otimes  x\wedge y \otimes  1 ; \\
  \psi( 1 \otimes  \xx\wedge\DD \otimes 1)
      = 1\otimes y\wedge E\otimes 1 ; \\
  \psi( 1 \otimes  \yy\wedge\DD\otimes  1)
      = - 1\otimes x\wedge E\otimes 1 ; \\
  \psi( 1 \otimes  \yy\wedge\EE \otimes  1)
      = 1 \otimes  x\wedge D \otimes  1 + x\wedge\xi ; \\
  \psi( 1 \otimes  \xx\wedge\EE \otimes  1)
      = - 1\otimes  y\wedge D\otimes 1 + \zeta ; \\
  \psi( 1 \otimes  \xx\wedge\yy \otimes  1)
      = -1\otimes D\wedge E\otimes 1 - \xi\wedge E ; \\[3pt]
  \psi( 1 \otimes  \EE \otimes  1)
      = 1 \otimes  x\wedge y\wedge D \otimes  1; \\
  \psi( 1 \otimes  \DD\otimes  1)
      = - 1\otimes x\wedge y\wedge E\otimes 1 ; \\
  \psi( 1 \otimes  \yy \otimes  1)
      = 1\otimes x\wedge D\wedge E\otimes 1 + x\wedge\xi\wedge E; \\
  \psi( 1 \otimes  \xx \otimes  1)
      = -1\otimes y\wedge D\wedge E\otimes 1 + \zeta\wedge E ; \\[3pt]
  \psi( 1\otimes   1)
      = 1\otimes x\wedge y\wedge D\wedge E\otimes 1,
  \end{gather}
where $\xi\in A\otimes V\otimes A$ and $\zeta\in A\otimes\Lambda^2V\otimes
A$ are chosen so that
  \begin{align}
   &d_1(\xi)
    = \tilde \nabla_y(F) -1|F_y,
   &&d_2(\zeta)
    =  \xi y - y \xi - 1 | y |F_y - \tilde \nabla _x^x(F) + \nabla(F).
  \end{align}
That there are elements which satisfy these two conditions follows
immediately from the exactness of the Koszul resolution of~$S$ as an
$S$-bimodule ---indeed, the right hand sides of the two conditions are
cycles in that complex--- but we can exhibit a specific choice:
if we write $F=\sum_{a+b=r+1}c_ax^ay^b$, with
$c_0$,~\dots,~$c_{r-1}\in\kk$, then we can pick
  \begin{align}
  &  \xi =\sum_{\substack{a+b=r+1\\s+t+1=b-1}}(t+1) c_a y^s|y|x^a y^t,
  && \zeta = \sum_{\substack{a+b=r+1\\s+t+1=b\\s'+t'+1=a}}
        c_ax^{s'}y^s|x\wedge y|x^{t'}y^t.
  \end{align}
That these formulas for~$\psi$ do indeed define a morphism of complexes
follows from a direct computation and it is easy to see that it is in fact
an isomorphism, as for an appropriate ordering of the bases of the
bimodules involved the matrices for the components of~$\psi$ are upper
triangular. Of course, it therefore induces an isomorphism in cohomology
and, since $A_\sigma$ is $A$-projective on the left, we conclude that there
are isomorphisms of $A$-bimodules
  \[
  H^i(\P^\vee) 
    \cong H^i(\P\otimes_A A_\sigma)
    \cong \begin{cases}
      A_\sigma  &\text{if $i=4$;} \\
      0 &\text{if $i>0$.}
      \end{cases}
  \]
This completes the proof.
\end{proof}

\pagebreak[4]

\section{Automorphisms, isomorphisms and normal elements}
\label{sect:autos}

\paragraph Our next objective is to compute the group of automorphisms of
the algebra~$A$. We start by describing some graded automorphisms of~$A$. Later
we will see that these are, in fact, \emph{all} the graded automorphisms of
our algebra, and that together with the exponentials of locally
$\ad$-nilpotent elements they generate the whole group $\Aut(A)$.

\begin{Lemma*}\plabel{lemma:iso:gr}
If $\begin{psmallmatrix}a&b\\c&d\end{psmallmatrix}\in\GL_2(\kk)$
and $e\in\kk^\times$ are such that
  \[ \label{eq:iso:cond}
  \frac{1}{(ad-bc)e}Q(ax+by,cx+dy) = Q(x,y),
  \]
and $v\in\kk$ and $\phi_0\in S_r$,
then there is a homogeneous algebra automorphism $\theta:A\to A$ such that
  \begin{align}
  &  \theta(x) = ax+by,
  && \theta(y) = cx+dy,
  && \theta(E) = E+v
  \end{align}
and
  \[ \label{eq:iso:D}
  \theta(D) =
    \begin{cases*}
    \phi_0 - \frac{ebF}{ax+by}E + eD, & if $b\neq 0$; \\
    \phi_0+eD, & if not.
    \end{cases*}
  \]
\end{Lemma*}

\begin{proof}
This is proved by a straightforward calculation. It should be noted that
the quotient appearing in the formula~\eqref{eq:iso:D} is always a
polynomial.
\end{proof}

\paragraph Recall that a \emph{higher derivation} of~$A$ is a sequence
$d=(d_i)_{i\geq0}$ of linear maps $A\to A$ such that $d_0=\id_A$
and for all $a$,~$b\in A$ and all $i\geq0$ we have the \emph{higher Leibniz
identity}
  \[
  d_i(ab) = \sum_{s+t=i} d_s(a)d_t(b).
  \]
It is clear that if $d=(d_i)_{i\geq0}$ is a higher derivation and $m\geq0$,
then the sequence $d^{[m]}=(d_i^{[m]})_{i\geq0}$ with
  \[
  d_i^{[m]} =
    \begin{cases*}
    d_{i/m}, & if $i$ is divisible by $m$; \\
    0, & if not
    \end{cases*}
  \]
is also a higher derivation. On the other hand,
if $d=(d_i)_{i\geq0}$ and $d'=(d'_i)_{i\geq0}$ are higher derivations
of~$A$, we can construct a new higher derivation $(d''_i)_{i\geq0}$, which
we denote $d\circ d'$, putting $d''_i = \sum_{s+t=i} d_s\circ d'_t$ for all
$i\geq0$. Finally, if $\delta:A\to A$ is a derivation of~$A$, then the
sequence $(\tfrac{1}{i!}\delta^i)_{i\geq0}$ is a higher
derivation, which we denote by $\exp(\delta)$; notice that this makes sense
because our ground field~$\kk$ has characteristic zero.

We let $D(A)$ be the associative subalgebra of~$\End_\kk(A)$ generated
by~$\Der(A)$, and say that two higher derivations $d=(d_i)_{i\geq0}$ and
$d'=(d'_i)_{i\geq0}$ of~$A$ are equivalent, and write $d\sim d'$, if for
all $i\geq0$ the map $d_i-d'_i$ is in the subalgebra
of~$\End_\kk(A)$ generated by~$D(A)$ and $d_0$,~\dots,~$d_{i-1}$; one can
check that this is indeed an equivalence relation on the set of higher
derivations.

\paragraph We recall the following very useful lemma from~\cite{AC}:

\begin{Lemma*}\plabel{lemma:ac}
If $d=(d_i)_{i\geq0}$ is a higher derivation of~$A$, then $d_i\in D(A)$ for
all $i\geq0$.
\end{Lemma*}

\begin{proof}
The result is an easy consequence of the fact that
  \[ \label{eq:iso:claim}
  \claim{if $d$ is a higher derivation of~$A$ and $j\geq1$, then there
  exists a higher derivation $d'=(d'_i)_{i\geq0}$ such that $d'\sim d$,
  $d'_i=0$ if $1<i<j$, and $d'_j$ is an element of~$\Der(A)$.}
  \]
To prove that this holds, let $d=(d_i)_{i\geq0}$ and suppose there is an $j\geq1$ such
that that $d_i=0$ if $1<i<j$. The higher Leibniz identity implies that
$d_{j}$ is an element of~$\Der(A)$, and then we can consider the higher
derivation $\exp(-d_{j})^{[j]}$. We let $d'=(d'_i)_{i\geq0}$ be the
composition $\exp(-d_{j})^{[j]}\circ d$. It is immediate that $d\sim d'$
and a simple computation shows that $d'_i=0$ if $1<i<j+1$. The
claim~\eqref{eq:iso:claim} follows inductively from this.
\end{proof}

\begin{Lemma}\label{lemma:iso:xy}
An element of~$A$ commutes with~$x$ and with~$y$ if and only if it belongs
to~$S$.
\end{Lemma}

\begin{proof}
The sufficiency of the condition is clear. To prove the necessity, let $e\in
A$ be such that $[x,e]=[y,e]=0$.
There are an integer $m\geq0$ and elements $\phi_0$,~\dots,~$\phi_m$ in the
subalgebra generated by~$x$,~$y$ and~$D$ in~$A$ such that
$e=\sum_{i=0}^m\phi_iE^i$, and we have
$0=[x,e_l]=\sum_{i=0}^m\phi_i\tau_1(E^i)$: this tells us that $\phi_i=0$ if
$i>0$, and that $e=\phi_0$. In particular, there are an integer $n\geq0$
and elements $\psi_0$,~\dots,~$\psi_n$ in~$S$ such that
$e=\sum_{i=0}^n\psi_iD^i$.
If $i\geq0$ we have $[D^i,y]\equiv iFD^{i-1}\mod\bigoplus_{j=0}^{i-2}SD^j$,
so that
  \[
  0 = [e,y] = \sum_{i=0}^n\psi_i[D^i,y]
    \equiv n\psi_nFD^{n-1} \mod \bigoplus_{j=0}^{n-2}SD^i.
  \]
Proceeding by descending induction we see from this that $\psi_i=0$ if
$i>0$, so that $e=\psi_0\in S$.
\end{proof}

\begin{Proposition}\label{prop:iso:id}
If $\theta:A\to A$ is an automorphism of~$A$ such that for all $i\geq0$ and
all $a\in A_i$ we have $\theta(a)\in a+\bigoplus_{j>i}A_j$, then here
exists an $f\in S$, uniquely determined up to the addition of a constant, such that
  \begin{align}
  &  \theta(x) = x,
  && \theta(y) = y,
  && \theta(D) = D-Ff_y,
  && \theta(E) = E-[E,f].
  \end{align}
Conversely, every $f\in S$ determines in this way an automorphism of~$A$
satisfying that condition.
\end{Proposition}

\begin{proof}
Let $\theta:A\to A$ be an automorphism of~$A$ as in the statement.
For each $j\geq0$ there is a unique linear map $\theta_j:A\to A$ of
degree~$j$ such that for each $i\geq0$ and each $a\in A_i$ the element
$\theta_j(a)$ is the $(i+j)$th homogeneous component of~$\theta(a)$.
We have that for all $a\in A$ we have $\theta_j(a)=0$ for $j\geq0$ and
$\theta(a)=\sum_{j\geq0}\theta_i(a)$ and, moreover, the sequence
$(\theta_j)_{j\geq0}$ is a higher derivation of~$A$. In particular, it
follows from Lemma~\pref{lemma:ac} that
  \[ \label{eq:iso:cc}
  \claim{$\theta_i\in D(A)$ for all $i\geq0$.}
  \]
We know, from Proposition~\pref{prop:hh}, that $\Der(A)=S_r\hat
D\oplus\kk\hat E\oplus\InnDer(A)$. If $u$ is an irreducible factor of $xF$,
then $(\phi\hat D)(uA)$, $\hat E(uA)$ and $[a,uA]$ are all contained
in~$uA$ for all $\phi\in S_r$ and all $a\in A$, and therefore
\eqref{eq:iso:cc} implies that that $\theta(uA)\subseteq uA$. As our
argument also applies to the inverse automorphism~$\theta^{-1}$, we have
$\theta^{-1}(uA)\subseteq uA$ and, therefore, $\theta(uA)=uA$. Since all units
of~$A$ are in~$\kk$, we see that $\theta(u)=u$. Since of~$xF$ has two
linearly independent linear factors, we can conclude that $\theta(x)=x$ and
$\theta(y)=y$.

Let $\theta(E)=E+e_1+\cdots+e_l$ with $e_i\in A_i$ for each
$i\in\{1,\dots,l\}$. We have
  \[
  x=\theta(x)=[\theta(E),\theta(x)]=[E,x]+[e_1,x]+\cdots+[e_l,x]
  \]
and, by looking at homogeneous components, we see that $[e_i,x]=0$ for all
$i\in\{1,\dots,l\}$ Similarly, $[e_i,y]=0$ for such~$i$, and therefore
Lemma~\pref{lemma:iso:xy} tells us that $e_1$,~\dots,~$e _l\in S$.

Suppose now that $\theta(D)=D+d_{r+1}+\cdots+d_l$ with $d_j\in A_j$ for
each $j\in\{r+1,\dots,l\}$. Considering the equality
$[\theta(E),\theta(D)]=r\theta(D)$ we see that $d_{r+i}=\tfrac1iFe_{iy}$
for each $i\in\{1,\dots,l\}$. Putting
$f=-\sum_{i=1}^l\tfrac1ie_i$, we obtain the first part of the lemma. The
second part follows from a direct verification.
\end{proof}

\paragraph The automorphisms described in Proposition~\pref{prop:iso:id}
are precisely the exponentials of the inner derivations corresponding to
locally $\ad$-nilpotent elements of~$A$. This  is a consequence of the
following result:

\begin{Proposition*}\plabel{prop:local-ad-nilp}
An element of~$A$ is locally $\ad$-nilpotent if and only if it belongs
to~$S$. If $f\in S$, then the automorphism $\exp\ad(f)$ maps $x$,~$y$,~$D$
and~$E$ to $x$,~$y$,~$D-Ff_y$ and $E-[E,f]$, respectively.
\end{Proposition*}

\begin{proof}
Suppose that $e\in A$ is a locally $\ad$-nilpotent element. The kernel
$\ker\ad(e)$ is a factorially closed subalgebra of~$A$, so that whenever
$a$,~$b\in A$ and $\ad(e)(ab)=0$ we have $\ad(e)(a)=0$ or $\ad(e)(b)=0$;
see \cite{Freudenburg} for the proof of this in the commutative case, which
adapts to ours.

Since $[x^iy^jD^kE^l,x]=-x^{i+1}y^jD^k\tau_1(E^l)$ for all
$i$,~$j$,~$k$,~$l\geq0$, we have $[A,x]\subseteq xA$ and from this we see
immediately that $[A,xA]\subseteq xA$. This implies that there is a
sequence $(u_k)_{k\geq0}$ in~$A$ such that $\ad(e)^k(x)=xu_k$ for all
$k\geq0$. Since $e$ is locally $\ad$-nilpotent, we can consider the
integer $k_0=\max\{k\in\NN_0:\ad(e)^k(x)\neq0\}$, and then we have
$0\neq xu_{k_0}\in\ker\ad(e)$. As $\ker\ad(e)$ is factorially closed, we see that
$\ad(e)(x)=0$. In other words, the element~$e$ commutes with~$x$.

There are an integer $m\geq0$ and elements $\phi_0$,~\dots,~$\phi_m$ in the
subalgebra generated by~$x$,~$y$ and~$D$ in~$A$ such that
$e=\sum_{i=0}^m\phi_iE^i$, and we have
$0=[x,e]=\sum_{i=0}^m\phi_i\tau_1(E^i)$: this tells us that $\phi_i=0$ if
$i>0$, and that $e=\phi_0$. In particular, there are an integer $n\geq0$
and elements $\psi_0$,~\dots,~$\psi_n$ in~$S$ such that
$e=\sum_{i=0}^n\psi_iD^i$.

An induction shows that $[D^i,F]\in FA$ for all $i\geq0$, and using this we
see that $[e,F]=\sum_{i=0}^n\psi_i[D^i,F]\in FA$, from which it follows
that in fact $[e,FA]\subseteq FA$. There is therefore a sequence
$(v_i)_{i\geq0}$ of elements of~$A$ such that $\ad(e)^i(F)=Fv_i$ for all
$i\geq0$. The local nilpotence of the map~$\ad(e)$ allows us to consider the integer
  \[
  i_0=\max\{i\in\NN_0:\ad(e)^i(F)\neq0\},
  \]
and then $0\neq
Fv_{i_0}\in\ker\ad(e)$. If $ax+by$ is any of the factors of~$F$, we have
$b\neq0$ and $ax+by\in\ker\ad(e)$: clearly, this implies that $y$ commutes
with~$e$. 

In view of Lemma~\pref{lemma:iso:xy}, we see that $e\in S$:
this proves the necessity of the condition
for local $\ad$-nilpotency given in the lemma. Its sufficiency is a direct
consequence of the fact that the graded algebra associated to the
filtration on~$A$ described in~\pref{p:start} is commutative. Finally, the
truth of the last sentence of the proposition can be verified by
an easy computation.
\end{proof}

\paragraph We write $\Aut_0(A)$ the set all automorphisms of~$A$ described
in Lemma~\pref{lemma:iso:gr}, and $\Exp(A)$ the set of all automorphisms
of~$A$ described in Proposition~\pref{prop:iso:id}; they are subgroups of
the full group of automorphisms~$\Aut(A)$.

\begin{Theorem}
The group $\Aut(A)$ is the semidirect product $\Aut_0(A)\ltimes\Exp(A)$,
corresponding to the action of~$\Aut_0(A)$ on $\Exp(A)$ given by
  \[
  \theta_0 \cdot \exp\ad(f) = \exp\ad(\theta^{-1}(f))
  \]
for all $\theta_0\in\Aut_0(A)$ and $f\in S$. The subgroup $\Aut_0(A)$ is
precisely the set of automorphisms of~$A$ preserving the grading and
$\Exp(A)$ is the set of exponentials of locally nilpotent inner
derivations of~$A$.
\end{Theorem}

\smallskip

Notice that the action described in this statement makes sense, as
$\theta_0(S)=S$ whenever $\theta_0$ belongs to~$\Aut_0(A)$.

\begin{proof}
Let $\theta:A\to A$ be an automorphism and let us write
$\theta(E)=e_0+\cdots+e_l$,
$\theta(x)=x_0+\cdots+x_l$,
$\theta(y)=e_0+\cdots+y_l$,
$\theta(D)=d_0+\cdots+d_l$ with $e_i$,~$x_i$,~$y_i$,~$d_i\in A_i$ for each
$i\in\{0,\dots,l\}$. Since $\theta$ is an automorphism, we have
  \begin{align} \label{eq:iso:comm}
  &  [\theta(E),\theta(x)] = \theta(x),
  && [\theta(E),\theta(y)] = \theta(y),
  && [\theta(E),\theta(D)] = r\theta(D).
  \end{align}
Looking at the degree zero parts of these equalities, and remembering that
$A_0$ is a commutative ring, wee see $x_0=y_0=d_0=0$. As $\theta(x)\neq0$,
we can consider the number $s=\min\{i\in\NN_0:x_i\neq0\}$ and we have
$s>0$. Looking that the component of degree~$s$ of the first equality
in~\eqref{eq:iso:comm}, we see that $[e_0,x_s]=x_s$. This means that the
restriction $\ad(e_0):A_s\to A_s$ has a nonzero fixed vector. Now $A_s$ as
a right $\kk[E]$-module is free with basis $\{x^iy^jD^k:i+j+rk=s\}$, the
map $\ad(e_0)$ is right $\kk[E]$-linear, and coincides with right
multiplication by $-\tau_s(e_0)$ on~$A_s$. Clearly, the existence of nonzero
fixed vector implies that $-\tau_s(e_0)=1$, so that $e_0=uE+v$ for some
$u\in\kk^\times$ and $v\in\kk$ with $su=1$.
Putting now $s'=\min\{i\in\NN_0:y_i\neq0\}$ and
$s''=\min\{u\in\NN_0:d_i\neq0\}$ and looking at the components in the least
possible degree in the second and third equations of~\eqref{eq:iso:comm},
we find that $s'u=1$ and $s''u=r$. In particular, $s=s'$ and $s''=rs$.

Suppose for a moment that $s>1$. As $\theta(x)$, $\theta(y)$ and
$\theta(D)$ are in the ideal $(A_s)$ generated by~$A_s$, the composition
$q:A\to A$ of $\theta$ with the quotient map $A\to A/(A_s)$ is a surjection
such that $q(A_0)=A/(A_s)$. This is impossible, as $A_0$ is a commutative
ring and $A/(A_s)$ is not: we therefore have $s=1$ and, as a consequence,
$u=1$.

There exist $a$,~$b$,~$c$,~$d\in\kk[E]$ such that $x_1=xa+yb$ and
$y_1=xc+yd$. The four elements $\theta(E)$, $\theta(x)$, $\theta(y)$
and $\theta(D)$ generate~$A$ and, as $\theta(D)$ is in $\bigoplus_{i\geq
r}A_i$, the elements $x$ and $y$ are in the subalgebra generated by the first three.
It follows at once that $x$,~$y\in x_1\kk[E]+y_1\kk[E]$ and, therefore,
that $\begin{psmallmatrix}a&b\\c&d\end{psmallmatrix}\in\GL_2(\kk[E])$.

Let us write $f\in\kk[E]\mapsto\vec f\in\kk[E]$ the unique algebra morphism
such that $\vec E=E+1$.
We have $[\theta(x),\theta(y)]=0$ and in degree~$2$ this tells us that
  \[
  x^2(a\vec c-\vec ac)
      + xy\,\something{T}
      + y^2(b\vec d-\vec bd)
      = 0,
  \]
so that 
  \begin{align}
  &  a\vec c=\vec ac,
  && b\vec d=\vec cd. \label{eq:acac}
  \end{align}
Suppose that
$a$ is not constant. As the characteristic of~$\kk$ is zero (and possibly
after replacing $\kk$ by an algebraic extension, which does not change
anything) there is then a $\xi\in\kk$ such that
$a(\xi)=0$ and $\vec a(\xi)=a(\xi+1)\neq0$, and the first
equality in~\eqref{eq:acac} implies that $c(\xi)=0$. The determinant of
$\begin{psmallmatrix}a&b\\c&d\end{psmallmatrix}$ is thus divisible by
$E-\xi$, and this is impossible. Similarly, we find that all of
$b$,~$c$,~$d$ must be constant.

Since $d_r\in A_r$, there exist $k\geq0$, $\phi_0$,~\dots,~$\phi_k\in S_r$ and
$h\in\kk[E]$ such that $d_r=\sum_{i=0}^k\phi_iE^i+Dh$. The component of
degree~$r+1$ of $[\theta(D),\theta(x)]$ is
  \[
  0  
    = [d_r,x_1]
    = - \sum_{i=0}^k(ax+by)\phi_i\tau_1(E^i)
      - (ax+by)D\tau_1(h)
      + bF\vec h.
  \]
We thus see that $h$ is constant, that $\phi_i=0$ if
$i\geq2$, and that 
  \[
  (ax+by)\phi_1+bhF=0.
  \]
If $b=0$, then $\phi_1=0$, and if
instead $b\neq0$, then either $h\neq0$ and we see that $ax+by$ divides~$F$ and that
$\phi_1=-bhF/(ax+by)$, or $h=0$ and $\phi_1=0$. 
In any case, we see that 
  \[
  d_r =
    \begin{cases*}
    \phi_0 - \frac{hbF}{ax+by}E + hD, & if $b\neq 0$; \\
    \phi_0+hD, & if not.
    \end{cases*}
  \]
Finally, the component of degree~$r+1$ of the equality
$[\theta(D),\theta(y)]=\theta(F)$ tells us that
  \[
  F(ax+by,cx+dy) = (ad-bc)h\frac{xF}{ax+by}.
  \]
It follows now from Lemma~\pref{lemma:iso:gr} that there is a graded automorphism
$\theta_0:A\to A$ such that $\theta_0(x)=ax+by$, $\theta_0(y)=cx+dy$,
$\theta_0(E)=E+v$ and $\theta_0(D)=d_r$.
The composition $\theta_0^{-1}\circ\theta$ satisfies the hypothesis of
Proposition~\pref{prop:iso:id}, and then there exists an~$f\in S$ such that
$\theta=\theta_0\circ\exp\ad(f)$. This shows that
$\Aut(A)=\Aut_0(A)\cdot\Exp(A)$. Moreover, if $\theta$~is a graded
automorphism, then so is $\exp\ad(f)=\theta_0^{-1}\circ\theta$ and,
since it maps $E$ to $E-[E,f]$,
this is possible if and only if $f\in\kk$, that is, if and only if
$\exp\ad(f)=\id_A$; this proves the last claim of the theorem.

Finally, computing the action of both sides of the equation on the generators
of~$A$, we see that
  \[
  \exp\ad(f)\circ\theta_0 = \theta_0\circ\exp\ad(\theta^{-1}(f))
  \]
for all $f\in S$ and all~$\theta_0\in\Aut_0(A)$, and this tells us that $\Aut(A)$ is
indeed a semidirect product $\Aut_0(A)\ltimes\Exp(A)$.
\end{proof}

\paragraph As usual, we say that an element~$u$ of~$A$ is
\newterm{normal} if $uA=Au$. Such an element, since it is not a
zero-divisor, determines an automorphism $\theta_u:A\to A$ uniquely by the
condition that $ua=\theta_u(a)u$ for all $u\in A$.

\begin{Proposition*}
Let $Q=\alpha_0\cdots\alpha_{r+1}$ be a factorization of~$Q$ as
a product of linear factors. The set of non-zero normal elements of~$A$ is
  \[
  \N(A) = 
  \{\lambda\alpha_0^{i_0}\cdots\alpha_{r+1}^{i_{r+1}}:
        \lambda\in\kk^\times,i_0,\dots,i_{r+1}\in\NN_0\}.
  \]
This set is the saturated multiplicatively closed subset of~$A$ or of~$S$
generated by $Q$.
\end{Proposition*}

\begin{proof}
A direct computation shows that each of the factors
$\alpha_0$,~\dots,~$\alpha_{r+1}$ of~$Q$ is normal in~$A$, so the
set~$\N(A)$ is contained in the set of normal elements of~$A$, for the
latter is multiplicatively closed. The set $\N(A)$ is 
multiplicatively closed and it is saturated because $S$ is closed under
divisors in~$A$, and it is clear that as a saturated multiplicatively
closed it is generated by~$Q$. To conclude the proof, we have to show
that every non-zero normal element of~$A$ belongs to~$\N(A)$.

Let $u$ be a normal element in~$A$ and let $\theta_u:A\to A$ be the
associated automorphism, so that $ua=\theta_u(a)u$ for all $a\in A$.
There are $k$,~$l\in\NN_0$ with $k\leq l$ and elements $u_k$,~\dots~$u_l\in
A$ such that $u=u_k+\cdots+u_l$, $u_i\in A_i$ if $k\leq i\leq l$, and
$u_k\neq0\neq u_l$. Similarly, there are $s$,~$t\in\NN_0$ with $s\leq t$
and elements $e_s$,~\dots,~$e_t\in A$ such that
$\theta_u(E)=e_s+\cdots+e_t$, $e_i\in A_i$ if $s\leq i\leq t$, and
$e_s\neq0\neq e_t$. As we have
  \[
  u_kE+\cdots+u_lE
    = uE
    = \theta_u(E)u 
    = e_su_k + \cdots + e_tu_l
  \]
with $u_kE$, $u_lE$, $e_su_k$ and $e_tu_l$ all non-zero, looking at the
homogeneous components of both sides we see that $s=t=0$. This means that
$\theta_u(E)=f(E)\in\kk[E]$, and therefore the above equality is really of the
form
  \[
  u_kE + \cdots + u_lE = f(E)u_k + \cdots + f(E)u_l.
  \]
It follows from this that $u_iE=f(E)u_i=u_if(E+i)$ for all
$i\in\{k,\dots,l\}$ and therefore that $E=f(E+k)$ and that $E=f(E+l)$.
Since our ground field has characteristic zero, this is only possible if
$k=l$: the element $u$ is homogeneous of degree~$l$.

Now, since $ua=\theta_u(a)u$ for all $a\in A$, the homogeneity of~$u$
implies immediately that $\theta_u$ is a homogeneous map. There are
$n\in\NN_0$ and $\phi_0$,~\dots,~$\phi_n$ in the subalgebra of~$A$
generated by~$x$,~$y$ and~$D$, such that $\phi_n\neq0$ and 
$u=\sum_{i=0}^n\phi_i E^i$. As
$\theta_u(x)$ has degree~$1$, it belongs to $S_1$ and we have
  \[
  \theta_u(x)\sum_{i=0}^n\phi_iE^i
    = \theta_u(x)u
    = ux
    = \sum_{i=0}^n\phi_iE^ix
    = x\sum_{i=0}\phi_i(E+1)^i.
  \]
Considering only the terms that have $E^n$ as a factor we see that
$\theta_u(x)=x$, and then the equality tells us that in fact
$\sum_{i=0}^n\phi_iE^i=\sum_{i=0}\phi_i(E+1)^i$. Looking now at the terms
which have $E^{n-1}$ as a factor here we see that moreover $n=0$, so that
$u\in\kk[x,y,D]$. There exist then $m\in\NN_0$ and
$\psi_0$,~\dots,~$\psi_m\in S$ such that $\psi_m\neq0$ and
$u=\sum_{i=0}^m\psi_iD^i$. As $\theta_u(y)$ has degree~$1$, it belongs
to~$S_1$ and we have
  \[
  \theta_u(y)\sum_{i=0}^m\psi_iD^i
    =\theta_u(y)u
    = uy
    = \sum_{i=0}^m\psi_iD^iy
    = \sum_{i=0}^my\psi_iD^i + \sum_{i=0}^m\psi_i[D^i,y].
  \]
Comparing the terms that have $D^m$ as a factor we conclude that also
$\theta_u(y)=y$.

As $\theta_u$ fixes~$x$ and~$y$, the element $u$ commutes with $x$ and~$y$,
and Lemma~\pref{lemma:iso:xy} allows us to conclude that $u$ is in~$S_l$.
Moreover, we know that all homogeneous automorphisms of~$A$ are those
described in Lemma~\pref{lemma:iso:gr}, so there exist $\phi\in S_r$ and
$e\in\kk^\times$ such that $\theta_u(D)=\phi+eD$. We then have that
  \[
  uD 
    = \theta_u(D)u 
    = (\phi+eD)u
    = \phi u + euD + eu_yF
  \]
and this implies that $e=1$ and $\phi u+u_yF=0$. Suppose now that $\alpha$
is a linear factor of~$u$ and let $k\in\NN$ and $v\in S$ be such that
$u=\alpha^kv$ and $v$ is not divisible by~$\alpha$. The last equality
becomes $\phi\alpha^kv+k\alpha^{k-1}\alpha_yvF+\alpha^kv_yF=0$ and implies
that $\alpha$ divides $\alpha_yF$: this means that $\alpha$ is a non-zero
multiple of~$x$ or a linear factor of~$F$. As $u$ can be factored as a
product of linear factors, we can therefore conclude that $u$ belongs to
the set described in the statement of the proposition.
\end{proof}

\paragraph There is a close connection between normal elements, the
first Hochschild cohomology space that we computed in
Section~\ref{sect:hh} and the modular automorphisms of~$A$.

\begin{Proposition*}\plabel{prop:conn}
Let $Q=\alpha_0\cdots\alpha_{r+1}$ be a factorization of~$Q$ as
a product of linear factors. 
\begin{thmlist}

\item Every linear combination of the derivations
$\partial_{\alpha_0}$,~\dots,~$\partial_{\alpha_{r+1}}:A\to A$
described in Proposition~\pref{prop:partial} is locally nilpotent.

\item If $u=\lambda \alpha_0^{i_0}\cdots\alpha_{r+1}^{i_{r+1}}$, with
$\lambda\in\kk^\times$ and $i_0$,~\dots,~$i_{r+1}\in\NN_0$, is a normal
element of~$A$, then the automorphism $\theta_u:A\to A$ associated to~$u$
is
  \[
  \theta_u = \exp\left(-\sum_{j=0}^{r+1}i_j\partial_{\alpha_{j}}\right).
  \]
This automorphism is such that $\theta_u(f) = f$ for all $f\in S$ and
  \[
  \theta_u(\delta) = \delta + \frac{\delta(u)}{u} 
  \]
for all $\delta\in\Der(\A)$.

\item The modular automorphism $\sigma:A\to A$ described in
Proposition~\pref{prop:CY} coincides with the automorphism
$\theta_{Q}$ associated to the normal element~$Q$.

\end{thmlist}
\end{Proposition*}

\paragraph Another immediate application of the determination of the set of
normal elements is the classification under isomorphisms of our algebras.

\begin{Proposition*}
Let $\A$ and $\A'$ be two central arrangements of lines in~$\AA^2$.
The algebras $D(\A)$ and $D(\A')$ are isomorphic
if and only if the arrangements $\A$ and~$\A'$ are isomorphic.
\end{Proposition*}

\begin{proof}
The sufficiency of the condition being obvious, we prove only its
necessity. We will denote with primes the objects associated to the
arrangement~$\A'$, so that for example $A'=D(\A')$ and so on. Moreover, in
view of the sufficiency of the condition we
can suppose without loss of generality that both arrangements $\A$
and~$\A'$ contain the line with equation $x=0$.

Let us suppose that there is an isomorphism of algebras $\phi:A\to A'$.
Since $\phi$ maps locally $\ad$-nilpotent elements to locally
$\ad$-nilpotent elements, it follows from
Proposition~\pref{prop:local-ad-nilp} that $\phi(S)=S'$ and therefore that
$\phi$ restricts to an isomorphism of algebras $\phi:S\to S'$. On the other
hand, $\phi$ also maps normal elements to normal elements, so that $\phi$
restricts to a monoid homomorphism $\phi:\N(A)\to\N(A')$. 
Let $Q=\alpha_0\cdots\alpha_{r+1}$ and $Q'=\alpha'_0\cdots\alpha'_{r'+1}$ be
the factorizations of~$Q$ and of~$Q'$ as products of linear factors.
The invertible elements of the monoid~$\N(A)$ are the units of~$\kk$ and
the quotient~$\N(A)/\kk^\times$ is the free abelian monoid generated by
(the classes of) $\alpha_0$,~\dots,~$\alpha_{r+1}$ and, of course, a similar
statement holds for the other arrangement. Since $\phi$ induces an
isomorphism $\N(A)/\kk^\times\to\N(A')/\kk^\times$ we see, first, that $r=r'$
and, second, that there are a permutation~$\pi$ of the set~$\{0,\dots,r+1\}$
and a function $\lambda:\{0,\dots,r+1\}\to\kk^\times$ such that
$\phi(\alpha_i)=\lambda(i)\alpha'_{\pi(i)}$ for all $i\in\{0,\dots,r+1\}$.
As there are at least two lines in each arrangement, this implies that the
restriction $\phi|_S:S\to S'$ restricts to an isomorphism of vector spaces
$\phi:S_1\to S'_1$, so that $\phi|_S$ is linear, and that $\phi(Q)=Q'$.
It is clear that this implies that the arrangements~$\A$ and~$\A'$ are
isomorphic.
\end{proof}

\paragraph A simple and final observation that we can make at this point is that our
algebra~$A$ and the full algebra~$\D(S)$ of regular differentials operators
of~$S$ are birational, that is, that they have the same fields of
quotients. In fact, the two algebras become isomorphic already after localization
at a single element:

\begin{Proposition*}\plabel{prop:birational}
The inclusion $A\to\D(S)$ induces after localization at~$Q$ an isomorphism
$A[\frac{1}{Q}]\to\D(S)[\frac{1}{Q}]$ and, in particular, $A$ and $\D(S)$
have isomorphic fields of fractions.
\end{Proposition*}

That both localizations actually exist follows from the usual
characterization of quotient rings; see, for example, \cite{MR}*{Chapter 2}.

\begin{proof}
Clearly the map $A[\frac{1}{Q}]\to\D(S)[\frac{1}{Q}]$ induced by the
inclusion is injective, and it is surjective since $S$ is contained in its
image as are $\partial_y=\tfrac{1}{F}D$ and $\partial_x=\tfrac{1}{x}E-\tfrac{y}{Q}D$.
\end{proof}


\vspace{1cm}

\end{document}